\documentclass[oneside, a4paper,11pt,reqno]{amsart}
\usepackage{color}

\usepackage{amssymb,amsmath,amsthm}
\usepackage[T1]{fontenc}

\newtheorem{hypo}{Hypothesis}

\newtheorem{prop}[hypo]{Proposition}

\newtheorem{thm}[hypo]{Theorem}

\newtheorem{lem}[hypo]{Lemma}

\newtheorem{coro}[hypo]{Corollary}

\def\A{\mathcal{A}}
\def\B{\mathcal{B}}
\def\C{\mathcal{C}}
\def\D{\mathcal{D}}
\def\N{\mathcal{N}}

\def\L{\mathcal{L}}
\def\M{\mathcal{M}}
\def\O{\mathcal{O}}
\def\E{\mathcal{E}}

\def\F{\mathcal{F}}
\def\V{\mathcal{V}}
\def\X{\mathcal{X}}

\def\PP{\mathbb{P}}
\def\RR{\mathbb{R}}
\def\ZZ{\mathbb{Z}}

\def\DD{\mathbb{D}}
\def\EE{\mathbb{E}}

\newcommand {\pare}[1] {\left( {#1} \right)}
\newcommand {\cro}[1] {\left[ {#1} \right]}
\def \ind {{\bf 1}}
\newcommand {\refeq}[1] {(\ref{#1})}
\newcommand {\va}[1] {\left| {#1} \right|}
\newcommand {\acc}[1] {\left\{ {#1} \right\}}
\newcommand {\floor}[1] {\left\lfloor {#1} \right\rfloor}

\newcommand {\nor}[1] { \left\| {#1} \right\|}
\newcommand {\bra}[1] {\left\langle {#1} \right\rangle}

\title[On the local time of random processes in random scenery]
       {On the local time of random processes in random scenery}
\author{Fabienne Castell} 
\address{LATP, UMR CNRS 6632. Centre de Math\'ematiques et Informatique.
Universit\'e Aix-Marseille I. 39, rue Joliot Curie. 13 453 Marseille Cedex
13. France.}
\email{Fabienne.Castell@cmi.univ-mrs.fr}

\author{Nadine Guillotin-Plantard} 
\address{Institut Camille Jordan, CNRS UMR 5208, Universit\'e de Lyon, Universit\'e Lyon 1, 43, Boulevard du 11 novembre 1918, 69622 Villeurbanne, France.}
\email{nadine.guillotin@univ-lyon1.fr}

\author{Fran\c{c}oise P\`ene}
\address{Universit\'e Europ\'eenne de Bretagne, Universit\'e de Brest,
Laboratoire de Math\'ematiques, UMR CNRS 6205, 29238 Brest cedex, France}
\email{francoise.pene@univ-brest.fr}

\author{Bruno Schapira}
\address{D\'epartement de Math\'ematiques, CNRS UMR 8628, B\^at. 425, Universit\'e Paris-Sud 11, F-91405 Orsay, cedex, France. }
\email{bruno.schapira@math.u-psud.fr}

\keywords{Random walk in random scenery; local limit theorem; local time; level sets\\
This research was supported by the french ANR project MEMEMO2}

\begin{document}

\begin{abstract} 
Random walks in random scenery are processes defined by  
$Z_n:=\sum_{k=1}^n\xi_{X_1+...+X_k}$, where basically $(X_k,k\ge 1)$ and $(\xi_y,y\in\mathbb Z)$
are two independent sequences of i.i.d. random variables. 
We assume here that $X_1$ is $\ZZ$-valued, centered and with finite moments of all orders. 
We also assume that $\xi_0$ is $\ZZ$-valued, centered and square integrable.  
In this case H. Kesten and F. Spitzer proved that 
$(n^{-3/4}Z_{[nt]},t\ge 0)$ converges in distribution as $n\to \infty$ toward some self-similar process $(\Delta_t,t\ge 0)$ called Brownian motion in random scenery. 
In a previous paper, we established that 
${\mathbb P}(Z_n=0)$ behaves asymptotically like a constant times $n^{-3/4}$,  
as $n\to \infty$. 
We extend here this local limit theorem: we give a precise asymptotic result for the probability for $Z$ to 
return to zero simultaneously at several times. 
As a byproduct of our computations, we show that $\Delta$ admits a bi-continuous version of its local time process which is  
locally H\"older continuous of order $1/4-\delta$ and $1/6-\delta$, respectively in the time and space variables, for any $\delta>0$. 
In particular, this gives a new proof of the fact, previously obtained by Khoshnevisan, 
that the level sets of $\Delta$ have Hausdorff dimension a.s. 
equal to $1/4$. We also get the convergence of every moment
of the normalized local time of $Z$ toward its continuous counterpart. 
\end{abstract}
\maketitle

\section{Introduction} 

\subsection{Description of the model and of some earlier results}
We consider two independent sequences $(X_k,k\ge 1)$ and $(\xi_y,y\in\mathbb Z)$
of independent identically distributed $\mathbb Z$-valued random variables. 
We assume in this paper that $X_1$ is centered, with finite moments of all orders,  
and that its support generates $\mathbb Z$.
We consider the {\it random walk} $(S_n, n \geq 0)$
defined by 
$$S_0:=0\quad \mbox{and}\quad S_n:=\sum_{i=1}^{n}X_i\quad  \textrm{for all } n\ge 1.$$
We suppose that $\xi_0$ is centered, with finite second moment $\sigma^2:=\EE[\xi_0^2]$. 
The sequence $\xi$ is called the {\it random scenery}.

\noindent The {\it random walk in random scenery} $Z$ is 
then defined for all $n\ge 1$ by
$$Z_n:=\sum_{k=0}^{n-1}\xi_{S_k}.$$ 
For motivation in studying this process and in particular for a description of its connections with many other models, we refer to \cite{BFFN, KestenSpitzer,LeDoussal} and references therein.  
Denoting by  
$N_n(y)$ the local time of the random walk $S$~:
$$N_n(y)=\#\{k=0,\dots,n-1\ :\ S_k=y\} \, ,
$$
it is straightforward, and important, to see that 
$Z_n$ can be rewritten as $Z_n=\sum_y\xi_yN_n(y)$.

\noindent Kesten and Spitzer \cite{KestenSpitzer} and Borodin \cite{Borodin}  proved 
the following functional limit theorem~:  
\begin{equation}
\label{eq1.08}
\left(n^{-3/4} Z_{nt},t\ge 0 \right)\quad  \mathop{\Longrightarrow}
_{n\rightarrow\infty}^{(\mathcal{L})} \quad  \left(\sigma\Delta_t,t\ge 0\right),
\end{equation} 
where
\begin{itemize} 
\item $Z_s:=Z_n+(s-n)(Z_{n+1}-Z_n)$, for all $n\leq s\leq n+1$, 
\item $\Delta$ is defined by
$$\Delta_t:=\int_{-\infty}^{+\infty} L_t(x)\, {\rm d}\beta_x\, ,$$
with $(\beta_x)_{x\in\mathbb R}$ a standard Brownian motion and 
$(L_t(x), \, t\ge 0,\ x\in\mathbb R)$ a jointly continuous in $t$ and $x$ version 
of the local time process of some other standard Brownian motion
$(B_t)_{t\ge 0}$ independent of $\beta$. 
\end{itemize} 
The process $\Delta$ is known to be a  
continuous $(3/4)$-self-similar 
process with stationary increments, and is called {\it Brownian motion in random scenery}. 
It can be seen as a mixture of stable processes, but it is not a stable process.

\noindent Let now $\varphi_\xi$ denote the characteristic function of $\xi_0$ and let $d$ be such 
that $\{u\, :\, |\varphi_\xi(u)|=1\} = (2\pi/d)\ZZ$.
In \cite{BFFN} we established the following local limit theorem~:
\begin{eqnarray}
\label{TLL1} 
\PP \pare{
Z_n=\floor{n^{\frac 34}x} } = \left\{ \begin{array}{ll} 
 d\sigma^{-1}\, p_{1,1}(x/\sigma)\, n^{-\frac 34}+o(n^{-\frac 34}) & \textrm{if } 
   {\mathbb P}\left(n\xi_0-
   \floor{n^{\frac 34}x}\in d\ZZ\right)=1 \\
0 & \textrm{otherwise},
\end{array} 
\right.
\end{eqnarray} 
with 
$$p_{1,1}(x):=\frac 1{\sqrt {2\pi}}\EE\left[||L_1||_2^{-1}e^{- 
      ||L_1||^2_2x^2/2}\right],$$
and 
$||L_1||_2:= \left(\int_\RR L_1^2(y)\, dy\right)^{1/2}$ the $L^2$-norm of $L_1$. 
In the particular case when $x=0$, we get
\begin{eqnarray}
\label{TLL1b} 
\PP \pare{
Z_n=0} = \left\{ \begin{array}{ll} 
d\sigma^{-1}\, p_{1,1}(0)\, n^{-\frac 34}+o(n^{-\frac 34}) & \textrm{if } 
   n\in d_0{\mathbb Z} \\
0 & \textrm{otherwise},
\end{array} 
\right.
\end{eqnarray} 
with $d_0:=\min\{m\ge 1\, :\, \varphi_\xi(2\pi/d)^m=1\}$.  
\footnote{Recall that, for every $n\ge 0$, we have
$${\mathbb P}(n\xi_0\in d{\mathbb Z})>0\ \ \Longleftrightarrow\ \  
   {\mathbb P}(n\xi_0\in d{\mathbb Z})=1 \ \ 
   \Longleftrightarrow\ \  n\in d_0{\mathbb Z}.$$}

\noindent Actually the results mentioned above were proved in the more general 
case when the distributions of the $\xi_y$'s and $X_k$'s are only supposed to be in the basin of 
attraction of stable laws (see \cite{FFN}, \cite{BFFN} and \cite{KestenSpitzer} for details).

\subsection{Statement of the results} 
\label{subsecintroresults}
\subsubsection{Local time of Brownian motion in random scenery} 
Let $T_1,\dots,T_k$, be $k$ positive reals. 
Set
$${\mathcal D}_{T_1,\dots,T_k}:=\det( M_{T_1,\dots,T_k})\ \ \mbox{with}\ \
    M_{T_1,\dots,T_k}=\left(\langle L_{T_i},L_{T_j}\rangle\right)_{1\le i,j\le k},$$
where $\langle \cdot,\cdot \rangle$ denotes the usual scalar product on $L^2(\RR)$, and 
$${\mathcal C}_{T_1,\dots,T_k}:= \EE\left[{\mathcal D}_{T_1,\dots,T_k}^{-1/2}\right].$$
Our first result is the following
\begin{thm}
\label{theoMk}
For any $k\ge 1$, there exist constants $c>0$ and $C>0$,  such that 
$$c\, \prod_{i=1}^k(T_i-T_{i-1})^{-3/4}\, \le \, \C_{T_1,\dots,T_k}\, \le \, C\, \prod_{i=1}^k(T_i-T_{i-1})^{-3/4},$$
for every $0<T_1<\dots<T_k$, with the convention that $T_0=0$. 
\end{thm}
The most difficult (and interesting) part here is the upper bound. 
The lower bound is obtained directly by 
using the scaling property of the local time of Brownian motion and the well-known Gram-Hadamard inequality. 
Concerning the upper bound, we will give more details about its proof in a moment, 
but let us stress already that even for $k=2$ the result is not immediate (whereas when $k=1$ it follows relatively easily 
from the Cauchy-Schwarz inequality and some basic properties of the Brownian motion, see for instance \cite{BFFN}).

\noindent A first corollary of this result is the following: 
\begin{coro}\label{densite}
For all $k\ge 1$ and all $0<T_1<\dots<T_k$, the random variable $(\Delta_{T_1},\dots,\Delta_{T_k})$ admits a continuous
density function, denoted by $p_{k,T_1,\dots,T_k}$, 
which is given by
$$p_{k,T_1,\dots,T_k}(x):=(2\pi)^{-\frac k2}\, 
    {\mathbb E}\left[ {\mathcal D}^{-1/2}_{T_1,\dots,T_k}\exp\left(
        -\frac 12\langle M_{T_1,\dots,T_k}^{-1}x,x\rangle\right)\right]\quad \textrm{for all }x\in\RR^k.$$
\end{coro} 
\noindent Theorem \ref{theoMk} also shows that, for every $t\ge 0$, $k\ge 1$
and $x\in{\mathbb R}$,
\begin{eqnarray}
 \label{Mk} 
\M_{k,t}(x):=  \int_{[0,t]^k} p_{k,T_1,\dots,T_k}(x,\dots,x) \, dT_1\dots dT_k,
\end{eqnarray} 
is finite. Define now the {\it level sets} of $\Delta$ as the sets of the form  
$$\Delta^{-1}(x):=
\{t\ge 0\ :\ \Delta_t=x\},$$ 
for $x\in \RR$. We can then state our main application of Theorem \ref{theoMk}, which can be deduced by standard techniques: 
\begin{thm} 
\label{thmLT}
There exists a nonnegative process $(\L_t(x),x\in \RR, t\ge 0)$, such that 
\begin{itemize} 
\item[(i)] a.s. the map $(t,x) \mapsto \L_t(x)$ is continuous and nondecreasing in $t$. Moreover for any $\delta>0$, it is locally H\"older continuous of order 
$1/4-\delta$, in the first variable, and of order $1/6-\delta$, in the second variable, 
\item[(ii)] a.s. for any measurable $\varphi : \RR \to \RR_+$, and any $t\ge 0$, 
$$\int_0^t \varphi(\Delta_s)\, ds = \int_{\RR} \varphi(x)\, \L_t(x)\, dx,$$
\item[(iii)] for any $T>0$, we have the scaling property: 
$$(\L_{t\, T}(x),t\ge 0,x\in \RR) \quad \mathop{=}^{(d)} \quad (T^{1/4}\, \L_t( x \, T^{-3/4}),t\ge 0,x\in \RR).$$ 
\item[(iv)] for any $x\in \RR$, $k\ge 1$, and $t>0$, the $k$-th moment of $\L_t(x)$ is finite and  
\begin{eqnarray}
 \label{momentLtz}
{\mathbb E}\left[\mathcal L_t(x)^k\right]=\M_{k,t}(x),
\end{eqnarray}
\item[(v)] a.s. for any $x\in \RR$, the support of the measure $d_t\L_t(x)$ 
is contained in $\Delta^{-1}(x)$. 
\end{itemize}
The random variable $\L_t(x)$ is called the local time of $\Delta$ in $x$ at time $t$. 
\end{thm}
We believe that the exponents $1/4$ and $1/6$ in Part (i) are sharp. One reason is that our proof gives the right critical exponents in the case of 
the Brownian motion. Another heuristic reason comes from a result proved by Dombry \& Guillotin \cite{DombryGuillotin},  
saying that the sum of $n$ i.i.d copies of the process  $\Delta$ converges under appropriate normalization,
towards a fractional Brownian motion with index $3/4$. But the H\"older continuity critical exponents of the local time of 
the latter process are exactly equal to $1/4$ in the time variable, and $1/6$ in the space variable.

Let us point out that an original feature of this theorem is that it gives strong regularity properties of the local time of a process which is neither Markovian nor Gaussian, 
whereas usually similar results are obtained when at least one of these conditions is satisfied (see for instance \cite{GH,MR}).

We should notice now that previously only the existence of a process satisfying (ii), \eqref{momentLtz} for $k\le 2$ and (v) was known, see \cite{YX}.
The original motivation of \cite{YX} 
was in fact to study the Hausdorff dimension of the level sets of $\Delta$. 
Khoshnevisan and Lewis conjectured in \cite{KL2} that their Hausdorff dimension was a.s. equal to $1/4$, for every $x\in \RR$.   
In \cite{YX} Xiao proved the result for almost every $x$, and left open the question to know 
whether this was true for every $x$.  
This has been later proved by Khoshnevisan in \cite{K}. 
With Theorem \ref{thmLT} we can now give an alternative proof, which follows the same lines as standard ones in the case of the Brownian motion ~: 
\begin{coro}[Khoshnevisan \cite{K}] 
\label{cor2}
For every $x\in \RR$, the Hausdorff dimension of $\Delta^{-1}(x)$ is a.s. equal to $1/4$. 
\end{coro}
Actually Xiao and Khoshnevisan proved their result in the more general setting where the Brownian motion $B$ is replaced by a stable process of index 
$\alpha \in (1,2]$. But at the moment it does not seem straightforward for us to adapt our proof to this case.

Now let us give some rough ideas of the proof of Theorem \ref{theoMk}. The first thing we use is that $M_{T_1,\dots,T_k}$ 
is a Gram matrix, and so there is nice formula for its smallest eigenvalue \eqref{ppvp}, which shows that to get a lower bound, 
it suffices to prove that the term $L_{T_1}/T_1^{3/4}$ is far in $L^2$-norm 
from the vector space generated by the terms $(L_{T_j}-L_{T_{j-1}})/(T_j-T_{j-1})^{3/4}$, for $j\ge 2$. Now by scaling we can always assume that $T_1=1$. 
Next by using the H\"older regularity of the process $L$, 
we can replace the $L^2$-norm by the $L^\infty$-norm, which is much easier to control.  
Then we use the Ray-Knight theorem, which says that, if instead of considering 
the term $L_1$ we consider $L_\tau$, with $\tau$ some appropriate random time, then we get a Markov process. It is then possible to prove 
that with high probability, this process is far in $L^\infty$-norm from any finite dimensional affine space, from which the desired result follows.

\subsubsection{Random walk in random scenery} 
Our first result is a multidimensional extension of our previous local limit theorem. 
We state it only for return probabilities to $0$, to simplify notation, but 
it works exactly the same is we replace $0$ by $[n^{3/4}x]$, for some fixed $x\neq 0$. 
\begin{thm}\label{thmTLL}
Let $k\ge 1$ be some integer and let $0<T_1<\dots<T_k$, be $k$ fixed positive reals. Then for any $n\ge 1$, 
\begin{itemize}
\item If $[nT_i]\in d_0\ZZ$, for all $i\le k$, then 
\begin{eqnarray*}
{\mathbb P} \pare{Z_{[nT_1]}= \dots=Z_{[nT_k]}=0} =
\left(d \sigma^{-1}\right)^k\, p_{k,T_1,\dots, T_k}(0,\dots,0)\ n^{-3k/4}+o(n^{-3k/4}). 
\end{eqnarray*}  
\item 
Otherwise $\PP\left(Z_{[nT_1]}= \dots=Z_{[nT_k]}=0\right) =0$.
\end{itemize}
\noindent Moreover, for every $k\ge 1$ and every $\theta\in(0,1)$, there exists $C=C(k,\theta)>0$, such that 
$$\PP\left[Z_{n_1}=\dots =
   Z_{n_1+\dots+n_k}=0\right]\le C\, (n_1\dots n_k)^{-3/4},$$
for all $n\ge 1$ and all $n_1,\dots,n_k\in [n^{\theta},n]$.
\end{thm}  
As an application we can prove that the moments of the local time of $Z$ converge toward their continuous counterpart. 
More precisely, for $z\in \ZZ$, define the local time of $Z$ in $z$ at time $n$ by: 
$${\mathcal N}_n(z):=\#\{m=1,\dots,n\, :\, Z_m=z\}.$$ 
Then Theorem \ref{thmTLL} together with the Lebesgue dominated convergence theorem give 
\begin{coro}\label{momenttpslocal}
For all $k\ge 1$, 
$${\mathbb E}\left[{\mathcal N}_n(0)^k\right]\, \sim\, 
   \left(\frac{d}{\sigma d_0}\right)^k \M_{k,1}(0) \, n^{k/4},$$
as $n\to \infty$, with $\M_{k,1}(0)$ as in \eqref{Mk}. 
\end{coro}
A natural question now is to know if we could not deduce from this corollary 
the convergence in distribution of the normalized local time $\N_n(0)/n^{1/4}$ toward $\L_1(0)$. 
To this end, we should need to know that the law of $\L_1(0)$ is determined by the sequence of its moments. 
Since this random variable is nonnegative, a standard criterion ensuring this, called Carleman's criterion, is the condition:  
$$\sum_k \M_{k,1}(0)^{-\frac 1{2k}} = \infty.$$ 
In particular a bound for $\M_{k,1}(0)$ in $k^{2k}$ would be sufficient. However with our proof, we only get a bound in $k^{ck}$, for some constant $c>0$. 
We can even obtain some explicit value for $c$, but unfortunately it is larger than $2$, so this is not enough to get the convergence in 
distribution. Note that 
this question is directly related to the question of the dependence in $k$ of the constant $C$ in Theorem \ref{theoMk}, which we 
believe is an interesting question for other problems as well, such as the problem of large deviations for the process $\L$ 
(see for instance \cite{CLRS} in
which the case of the fractional Brownian motion is considered).

Another interesting feature of Theorem \ref{thmTLL} is that it gives an effective measure of the asymptotic correlations of the increments of $Z$. 
Indeed, if we assume to simplify that $k=2$, $\sigma=1$ and $d=1$, then \eqref{TLL1} and Theorem \ref{thmTLL} (actually its proof) 
show that 
\begin{eqnarray}
\label{asympcor}
\frac{\PP(Z_{n+m}-Z_{n} = 0 \mid Z_{n}=0)}{\PP(Z_{n+m}-Z_{n} = 0)}\quad  
\longrightarrow \quad \frac {\EE\left[\left\{||L_1||_2^2||\widetilde L_t||_2^2-\langle L_1,\widetilde L_t\rangle^2\right\}^{-1/2}\right]}
{\EE\left[\left\{||L_1||_2||\widetilde L_t||_2\right\}^{-1}\right]},
\end{eqnarray}
as $n\to \infty$ and $m/n\to t$, for some $t>0$, where $L$ and $\widetilde L$ are the local time processes of two independent
standard Brownian motions. 
In particular the limiting value in \eqref{asympcor} is larger than one, 
which means that the process is asymptotically more likely to come back to $0$ at time $n+m$, 
if we already know that it is equal to $0$ at time $n$.

The general scheme of the proof of Theorem \ref{thmTLL} is quite close from the one used for the proof of \eqref{TLL1b} in \cite{BFFN}. 
However, in addition to Theorem \ref{theoMk} which is needed here and which is certainly the main new difficulty, some 
other serious technical problems appear in the multidimensional setting. In particular at some point we use   
a result of Borodin \cite{Bor} giving a strong approximation of the local time of Brownian motion by the random walk local time. 
This also explains why we need stronger hypothesis on the random walk here. 
Now concerning the scenery, it is not clear if we can relax the hypothesis of finite second moment, 
since we strongly use that the characteristic function of $(\Delta_{T_1},...,\Delta_{T_k})$, takes the form 
$$\psi(\theta_1,...,\theta_k)={\mathbb E}\left[e^{-\sum_{i,j=1}^k  a_{i,j}\theta_i\theta_j}\right],$$
with $(a_{i,j})_{i,j}$ some (random) positive symmetric matrix.

Finally let us mention that in the proof of Theorem \ref{thmTLL}, we use the following result, which might be interesting
on its own. It is a natural multidimensional extension of a result of Kesten and Spitzer \cite{KestenSpitzer} on the convergence 
in distribution of the 
normalized self-intersection local time of the random walk.
\begin{prop}\label{thmconvergencejointe}
Let $k\ge 1$ be given and let $T_1<\dots< T_k$, be $k$ positive reals. Then 
$$ \left(n^{-3/2}\langle N_{n_1+\dots+n_i},N_{n_1+\dots+n_j}\rangle\right)_{1\le i,j\le k} \quad 
\mathop{\Longrightarrow}^{(\mathcal{L})} \quad  \left( \langle L_{T_i},L_{T_j}\rangle\right)
_{1\le i,j\le k},$$
as $n\to \infty$, and $(n_1+\dots+n_i)/n\to T_i$, for all $i\ge 1$, and   
where for all $p,q$,  
$$\langle N_{p},N_{q}\rangle:=
    \sum_{y\in\mathbb Z}N_{p}(y)N_{q}(y).$$ 
\end{prop}

The paper is organized as follows. In Section \ref{preuve-densite}, we give a short proof of Corollary \ref{densite}. In Section \ref{CTfini} we prove Theorem \ref{theoMk}. 
Then in Section \ref{secLT}, we explain how one can deduce Theorem \ref{thmLT} and Corollary \ref{cor2} from it. 
Section \ref{secTLL} is devoted to the proof of Theorem \ref{thmTLL} and Section \ref{sectpsloc} to the proof of Corollary \ref{momenttpslocal}. 
Finally, in Section \ref{preuvecvgcejointe}, 
we give a proof of Proposition \ref{thmconvergencejointe}.

We also mention some notational convention that we shall use: if $X$ is some random variable and $A$ some set, then $\EE[X,\, A]$ will mean $\EE[X{\bf 1}_{A}]$. 
%%%%%%%%%%%%%%%%%%%%%%%%%%%%%%%%%%%%%%%%%%%%%%%%%%%%%%%%%%%%%%%%%%%%%%%%%%%%%%%%%%%%%%%%%%%%%%%%%%%%%%%%%%%%%%%%%%%%%%%%%%%%%%%%%%%%%%%%%%%%%%%%%%%%%%%%%%%%%%%%%%
%
%
\section{Proof of corollary \ref{densite}}\label{preuve-densite}
Let $k\ge 1$ be given and let $T_1 < \dots < T_k$, be some positive reals.
The characteristic function $\psi_{T_1,\dots,T_k}$ of 
$(\Delta_{T_1},\dots,\Delta_{T_k})$ (with the convention $T_0=0$)
is given by
\begin{eqnarray*}
 \psi_{T_1,\dots,T_k}(\theta) 
   &=& {\mathbb E}\left[
   \exp\left(-\frac 12\int_{\mathbb R}\left(\sum_{i=1}^k\theta_i L_{T_i}(u)
   \right)^2\, du\right)\right]\\
 &=& {\mathbb E}\left[
   \exp\left(-\frac 12\langle M_{T_1,\dots,T_k} \theta,\theta\rangle\right)\right],
\end{eqnarray*}
with $\theta:=(\theta_1,\dots,\theta_k)$. In particular this function is non-negative. Moreover, a change of variables  
(this change is 
possible since $\mathcal D_{T_1,...,T_k}$ is almost surely non null, thanks to
Theorem \ref{theoMk}) gives 
\begin{eqnarray*} 
\int_{\RR^k} \psi_{T_1,\dots,T_k}(\theta)\, d\theta
 &=& \C_{T_1,\dots,T_k} \,  
    \int_{\RR^k}e^{-\frac 12\sum_{i=1}^ku_i^2}\, du_1\dots du_k\\
 &=& (2\pi)^{\frac k2}\, \C_{T_1,\dots,T_k}<\infty.
\end{eqnarray*}
This implies, see the remark following Theorem 26.2 p.347 in \cite{Bil}, that
$(\Delta_{T_1},\dots,\Delta_{T_k})$ admits a continuous density function 
$p_{k,T_1,\dots,T_k}$, given by
\begin{eqnarray*}
p_{k,T_1,\dots,T_k}(x)
 &=&\frac 1{(2\pi)^{ k}}\int_{{\mathbb R}^k} e^{-i\langle \theta,x\rangle}
     \psi_{T_1,\dots,T_k}(\theta)\, d\theta\\  
&=&(2\pi)^{-\frac k2}\, {\mathbb E}\left[  \D^{-\frac 12}_{T_1,\dots,T_k}\exp\left(
        -\frac 12\langle M_{T_1,\dots,T_k}^{-1}x,x\rangle\right)\right],
\end{eqnarray*}
which was the desired result. \hfill $\square$

\section{Proof of Theorem \ref{theoMk}}

\label{CTfini}
Let $k\ge 1$ and $0<T_1<...<T_k$, be given. Set $t_i:=T_i-T_{i-1}$, for $i\le k$, with the convention that $T_0=0$.
For every $i=1,...,k$, let $(L^{(i)}_t(x):=L_{T_{i-1}+t}(x)-L_{T_{i-1}}(x),
t\in [0,t_i],x\in \RR)$ 
be the local time process of 
$B^{(i)}:=(B_{T_{i-1}+t},t\in[0,t_i])$. 
Set 
$$\widetilde \D_{t_1,\dots,t_k}:=\det(\widetilde M_{t_1,\dots,t_k}) \quad \textrm{with}\quad \widetilde M_{t_1,\dots,t_k}:=\left(\langle L^{(i)}_{t_i},L^{(j)}
_{t_j}\rangle\right)_{1\le i,j\le k},$$
and 
$$\widetilde {\mathcal C}_{t_1,\dots,t_k}:= \EE\left[\widetilde {\mathcal D}_{t_1,\dots,t_k}^{-1/2}\right].$$
Since $\widetilde {\mathcal D}_{t_1,\dots,t_k}={\mathcal D}_{T_1,...,T_k}$, 
Theorem \ref{theoMk} is equivalent to proving the existence of constants $c>0$ and $C>0$, such that
\begin{eqnarray}
 \label{thm1bis}
c\, (t_1\dots t_k)^{-3/4} \, \le \, \widetilde {\mathcal C}_{t_1,\dots,t_k}\, \le C \, (t_1\dots t_k)^{-3/4},
\end{eqnarray} 
for all positive $t_1,\dots,t_k$.

\noindent Let us first notice that $\widetilde {\mathcal D}_{t_1,\dots,t_k}$ is a Gram determinant and is thus nonnegative. 
So $\widetilde {\mathcal C}_{t_1,\dots,t_k}$ is well defined as an extended real number.

\noindent Now we start with the lower bound in \eqref{thm1bis}. We use the well known Gram-Hadamard inequality: 
$$\widetilde \D_{t_1,\dots,t_k} \, \le\,  \prod_{i=1}^k \, ||L_{t_i}^{(i)}||_2^2.$$
By using next the scaling property of Brownian motion, we see that 
$(t_i^{-3/4}||L_{t_i}^{(i)}||_2, i\ge 1)$ is a sequence of i.i.d. random variables 
distributed as $||L_{1}||_2$.  Therefore, 
$$\EE\left[\widetilde \D_{t_1,\dots,t_k}^{-1/2}\right] \, \ge\, c \, (t_1\dots t_k)^{-3/4},$$
with $c:= (\EE[||L_1||_2^{-1}])^k > 0$.  

\noindent We prove now the upper bound in \eqref{thm1bis}, which is the most difficult part. For this purpose,
we introduce the new Gram matrix
\[\overline{M}_{t_1,\dots,t_k} := \pare{\bra{ t_i^{-\frac 34}\, L^{(i)}_{t_i},t_j^{-\frac 34}\, L^{(j)}_{t_j} }}_{i,j}. 
\]
Note that all its eigenvalues are nonnegative and denote by $\overline{\lambda}_{t_1,\dots,t_k}$
the smallest one. We get then
\[ \D_{T_1,\dots,T_k} = \widetilde{\D}_{t_1,\dots,t_k}= \pare{\prod_{i=1}^k t_i^{3/2}} \,  \det(\overline{M}_{t_1,\dots,t_k})
\ge  \pare{\prod_{i=1}^k t_i^{3/2}} \,  \overline{\lambda}^k_{t_1,\dots,t_k} \, .
\]
Thus we can write 
\begin{eqnarray*}
\widetilde \C_{t_1,\dots,t_k} &=& \EE\left[\widetilde {\mathcal D}_{t_1,\dots,t_k}^{-1/2}\right] 
\\
&\le & (t_1\dots t_k)^{-3/4}\, \EE\left[\overline{\lambda}_{t_1,\dots,t_k}^{-k/2}\right] \\
&=& (t_1\dots t_k)^{-3/4}\, \int_0^\infty \PP\left[\overline{\lambda}_{t_1,\dots,t_k}^{-k/2} \ge t\right]\, dt\\
&\le & (t_1\dots t_k)^{-3/4}\, \left\{ 1+ \frac{2}{k}
\int_0^1 \PP[\overline{\lambda}_{t_1,\dots,t_k} \le \varepsilon] \, \frac{d\varepsilon}{\varepsilon^{1+k/2}}\right\}.
\end{eqnarray*}
Therefore Theorem \ref{theoMk} follows from the following proposition: 
\begin{prop} 
\label{proplambda}
For any $k\ge1$ and $K > 0$, there exists a constant $C>0$, such that 
$$\PP(\overline{\lambda}_{t_1,\dots,t_k} \le \varepsilon) \le C\, \varepsilon^K,$$
for all $\varepsilon\in (0,1)$ and all $t_1, \dots , t_k>0$. 
\end{prop}

\subsection{Proof of Proposition \ref{proplambda}.} 
Note first that 
\begin{eqnarray}
\label{ppvp}
\overline{\lambda}_{t_1,\dots,t_k} = \inf_{u_1^2 + \dots +u_k^2=1} \nor{u_1 t_1^{-\frac 34}\, L^{(1)}_{t_1} + \dots
+ u_k t_k^{-\frac 34}\, L^{(k)}_{t_k}  }_2^2 
\, .
\end{eqnarray}
Note next that if $u_1^2+ \dots +u_k^2=1$, then $u_{\max } :=\max_i \va{u_i} \ge  1/\sqrt{k}$. Thus dividing all $u_i$ by 
$u_{\max }$ leads to 
\[ \overline{\lambda}_{t_1,\dots,t_k} \ge \frac{1}{k} 
\min_{i=1, \dots, k} \,  \inf_{(v_j)_{j \ne i},\, \va{v_j} \le 1} \nor{t_i^{-\frac 34}\, L^{(i)}_{t_i}  + 
\sum_{j \ne i}  v_j t_j^{-\frac 34}\, L^{(j)}_{t_j} }_2^2
\, .
\]
Hence, it suffices to bound all terms 
\begin{equation}
\label{titj}
\PP \cro{ \inf_{(v_j)_{j \ne i},\, \va{v_j} \le 1} \nor{t_i^{-\frac 34}\, L^{(i)}_{t_i}  + 
\sum_{j \ne i}  v_j t_j^{-\frac 34} \, L^{(j)}_{t_j} }_2^2 \le k \varepsilon}, 
\end{equation}
for $i\le k$. By scaling invariance, and changing $t_j$ by $t_j/t_i$ in \eqref{titj}, one can always assume that $t_i=1$. 
It will also be no loss of generality to assume that 
$i=1$, the case $i>1$ being entirely similar. We are thus led to prove 
that for any $k \ge 1$ and $K > 0$, there exists a constant $C>0$, such that for all $\varepsilon \in (0,1)$, and all $t_j > 0$,  
\begin{equation}
\label{problambdabar}
\PP \cro{ \inf_{(v_j)_{j>1},\, |v_j|\le 1} \nor{L^{(1)}_1 + 
\sum_{j \ge 2}  v_j t_j^{-3/4}\, L^{(j)}_{t_j} }_2^2 \le \varepsilon} \le C\, \varepsilon^K \, .
\end{equation}
We want now to bound from below the $L^2$-norm by (some power of) the $L^{\infty}$-norm using the H\"older regularity of the Brownian local time.
To this end, notice that by scaling the constants  
$$C_H^{(j)}: = \sup_{x \ne y} \frac{\va{L^{(j)}_{t_j}(x) - L^{(j)}_{t_j}(y)}}{t_j^{3/8}\va{x-y}^{1/4}},$$
for $j\ge 1$, are i.i.d. random variables. Moreover, 
the constant of H\"older  continuity of order $1/4$ of the $j$-th term of the sum in \eqref{problambdabar} is 
larger than or equal to $C_H^{(j)} t_j^{-3/8}$. 
Since this can be large, we distinguish between 
indices $j$ such that $t_j$ is small from the other ones. More precisely, we define  $J=\{j :\ t_j\le \varepsilon^4\}$, and 
$$ \E_J : = \cup_{j\in J} \ \textrm{supp}(L_{t_j}^{(j)}),$$
where $\textrm{supp}(f)$ denotes the support of a function $f$. Set also 
$$\E'_J:= \{x\in \RR\ :\ d(x,\E_J) <  \varepsilon\}  \, . $$
%and for $R\ge 1$, 
%$$\E_J(R):=\{x\in \RR\ :\ d(x,\E'_J) < R\varepsilon\}.$$
To simplify notation, set now for all $x\in \RR$, and $v=(v_j)_{j \ge 2}$, 
$$F_v(x):= L^{(1)}_1(x) + \sum_{j\notin J,\, j \ge 2} v_j t_j^{-3/4} L^{(j)}_{t_j}(x) \quad \textrm{and} \quad G_v(x): = \sum_{j\in J} v_j t_j^{-3/4} L^{(j)}_{t_j}(x). $$
Notice that $G_v =0$ on $\E_J^c$ and that 
\[ \sup_{x \ne y} \frac{\va{F_v(x)-F_v(y)}}{\va{x-y}^{1/4}} \le \varepsilon^{-3/2} \sum_j C_H^{(j)}.
\]
Thus if for some $x\notin \E'_J$, $|F_v(x)|\ge \varepsilon$, and if in the same time $\sum_j C_H^{(j)} \le 1/(2\varepsilon^{1/4})$,
then 
\[ \nor{F_v+G_v}_2^2 \ge \int_{x-\varepsilon^{11}}^{x+\varepsilon^{11}} F_v(y)^2 \, dy 
\ge \int_{x-\varepsilon^{11}}^{x+\varepsilon^{11}} \left(\varepsilon - \varepsilon^{-3/2} \sum_j C_H^{(j)} \varepsilon^{11/4}\right)_+^2 \, dy 
\ge  \frac{1}{2} \varepsilon^{13} \, .
\]
Moreover, it is  known that the $C_H^{(j)}$ have finite moments of any order. 
Therefore it suffices to prove that for any $k\ge 1$ and $K> 0$, there is a constant $C>0$, such that for all $\varepsilon \in (0,1)$, 
and all $t_2, \dots, t_k>0$, 
\begin{eqnarray}
\label{normeinfini}
\PP\left(\inf_{v\in \RR^{k-1}} \, \sup_{x\notin \E'_J} \, |F_v(x)|\le \varepsilon\right) \le C\, \varepsilon^K \, . 
\end{eqnarray}
This will follow from the next two lemmas, that we shall prove in the next subsections: 
\begin{lem}
\label{lemeqx1}
Let $(L_t(x),t\ge 0, x\in \RR)$ be a continuous in $(t,x)$ version of 
the local time process of a standard Brownian motion $B$. 
Then for any $K > 0$ and $k \ge 0$, there exist $N\ge 1$ and $C>0$, such that, for any  
$\varepsilon \in (0,1)$, one can find  $N$ points $x_1,\dots,x_N \in \RR$, satisfying $\va{x_i-x_j} \ge \varepsilon^{1/8}$ for all $i\neq j$, and
\begin{eqnarray}
\label{4pL1}
\PP\left(\#\{j\le N\ : \ L_1(x_j)>\varepsilon^{1/4}\} \le k \right) \le C\, \varepsilon^K \, .
\end{eqnarray}
\end{lem}

\begin{lem} 
\label{lemeqx2}
For any $K>0$ and $k \ge 1$, there exist a constant $C>0$ and an integer $M \ge 1$, 
such that for all $x\in \RR$, $\varepsilon \in (0,1)$, 
and $t_2, \dots, t_k > 0$, 
\begin{eqnarray*}
\PP\left(L_1^{(1)}(x)>\varepsilon^{1/4},\  \inf_{v\in \RR^{k-1}}\, \sup_{|y-x|\le M\varepsilon}\, |F_v(y)| 
\le \varepsilon\right) \le C\, \varepsilon^K . 
\end{eqnarray*}
\end{lem}
Indeed, we can first always assume that $\E_J$ is included in the union of at most $k$ intervals of length $\varepsilon$, since for
any $j \in J$ and $K\ge 1$, by scaling there exists $C>0$, such that
\begin{equation}
\label{contEJ} \PP \cro{\sup_{s \le t_j} \va{B^{(j)}_s - B^{(j)}_0} \ge \varepsilon/2} \le \PP \cro{ \sup_{s \le \varepsilon^4} 
\va{B_s}  \ge \varepsilon/2 } = \PP \cro{ \sup_{s \le 1} 
\va{B_s}  \ge \varepsilon^{-1}/2 } \le C\, \varepsilon^K \, .
\end{equation}
Thus, among any $k+1$ points at distance larger than $\varepsilon^{1/8}$ from each other, at least one of them must be 
at distance larger than $M \varepsilon$ from $\E'_J$, at least if $\varepsilon$ is small enough. 
Therefore Lemma \ref{lemeqx1} shows that for any $K\ge 1$, there exists $C>0$, such that 
\begin{eqnarray*}
\PP \cro{ \inf_v \sup_{y \notin \E'_J} \va{F_v( y)} \le \varepsilon}
 \le C \, \varepsilon^K 
 + \sum_{m=1}^N \PP \cro{L_1^{(1)}(x_m) > \varepsilon^{1/4},\ \inf_v \sup_{\va{y-x_m} \le M \varepsilon}
  \va{F_v(y)} \le \varepsilon}
 \, , 
\end{eqnarray*} 
where $(x_1,\dots,x_N)$ are given by Lemma \ref{lemeqx1}. Then \eqref{normeinfini} follows from the above inequality and Lemma \ref{lemeqx2}. This concludes the proof of Proposition \ref{problambdabar}.  
\hfill $\square$

\subsection{Proof of Lemma \ref{lemeqx1}.} 
We first prove the result for $k=0$. Assume without loss of generality that $K$ is an integer larger than $1$ and  set 
$$\X_0=\acc{j\varepsilon^{1/8}\ :\ -8K \le j \le 8K}.$$
Set also $s_0:=0$ and for every  $m\ge 1$,   
$$s_m := \inf\{s>0\ :\ |B_s| \ge  m\, \varepsilon^{1/8}\}.$$ 
Note already that there exists $C>0$, such that for all $\varepsilon \in (0,1)$, 
\begin{eqnarray*}
\PP(s_{8K} >1) \le \PP\left(\sup_{s\le 1}|B_s| \le 8K \varepsilon^{1/8}\right) \le C \, \varepsilon^K,
\end{eqnarray*} 
by using for instance \cite[Proposition 8.4 p.52]{PS}. 
Thus it suffices to prove that 
\begin{equation}
\label{kzero}
\PP \left( L_{s_{8K}}(x) \le \varepsilon^{1/4} \quad \forall x \in \X_0\right)
\le \varepsilon^K , 
\end{equation}
for all $\varepsilon \in (0,1)$. By using the Markov property, and noting that $s_m \ge s_{m-1}+ s_1 \circ \theta_{s_{m-1}}$ (where $\theta$ is the usual
 shift on the trajectories),  we get a.s. for every $m \ge 1$,  
\[
 \PP\left(L_{s_m}(B_{s_{m-1}}) - L_{s_{m-1}}(B_{s_{m-1}})\le \varepsilon^{1/4}
    \mid \F_{s_{m-1}} \right) \le   \PP(L_{s_1}(0)\le \varepsilon^{1/4}) .
    \]
 By  the scaling property of the Brownian motion, we know that $L_{s_1}(0)$
 has the same law as $\varepsilon^{1/8} L'_1(0)$, 
with $L'_1(0)$ the local time of a standard Brownian motion taken at the first hitting time of $\{\pm 1\}$.   
Moreover, it is known that $L'_1(0)$ is an exponential random variable with parameter 1 (see for instance \cite{RY} Exercise (4.12) chap VI, p. 265). 
Therefore, a.s. for every $m\ge 1$,  
\begin{eqnarray*}
 \PP\left(L_{s_m}(B_{s_{m-1}}) - L_{s_{m-1}}(B_{s_{m-1}})\le \varepsilon^{1/4}
    \mid \F_{s_{m-1}} \right)
  \le\  \PP(L'_1(0)\le \varepsilon^{1/8})  \le  \varepsilon^{1/8} . 
\end{eqnarray*}
Then we get by induction, 
\[ 
 \PP\left( L_{s_{8K}}(x) \le \varepsilon^{1/4} \quad \forall x \in \X_0\right)
 \le 
  \PP\left(L_{s_m}(B_{s_{m-1}}) - L_{s_{m-1}}(B_{s_{m-1}}) \le \varepsilon^{1/4} \quad \forall m\le 8K\right)
\le \varepsilon^{K},
\]
proving \refeq{kzero}. This concludes the proof of the lemma for $k=0$.

\noindent Now we prove the result for general $k\ge 0$. For $m \in \ZZ$, consider the set 
$$\X_m:= m(16K+1) \varepsilon^{1/8} + \X_0.$$
Then the proof above shows similarly that for any $0\le m\le k$, 
\begin{equation*}
\PP \left( L_1(x) \le \varepsilon^{1/4} \quad \forall x \in \X_m \cup \X_{-m}\right)\le C\, \varepsilon^K .
\end{equation*}
The lemma follows immediately. \hfill $\square$

\subsection{Proof of Lemma \ref{lemeqx2}.} 
Let $K>0$ be fixed, and assume without loss of generality that $x\ge 0$. Fix also $M\ge 1$ some integer to be chosen later.

\noindent For every affine subspace $V$ of $\RR^M$, we denote by $V_\varepsilon$ the set 
$$V_\varepsilon:=\{v\in \RR^M \ :\ d(v,V)\le \varepsilon\},$$
where $d(v,V)=\min\{|v-y|_\infty\ :\ y\in V\}$. Then we can write 
\begin{eqnarray*} 
&& \PP\left(L^{(1)}_{1}(x) > \varepsilon^{1/4},\  
\inf_{v\in \RR^{k-1}} \sup_{\va{y-x} \le M \varepsilon } |F_v(y) |\le \varepsilon  \right)\\ 
& \le & \PP\left[L^{(1)}_{1}(x) >\varepsilon^{1/4},\ (L^{(1)}_{1}(x+\varepsilon),\dots,L^{(1)}_{1}(x+M\varepsilon))\in \V_\varepsilon\right]
:= {\mathcal P}_\varepsilon, 
\end{eqnarray*} 
where
$$\V:=Vect \left( \left(L^{(j)}_{t_j}(x+\ell\varepsilon )\right)_{
\ell=1,\dots,M},\ j \in I \right),$$
with $I:= \{j>1 \ : \ j \notin J\}$. 
Set now
$$\tau:=\inf \{s>0\ :\ L^{(1)}_s(x)>\varepsilon^{1/4}\},$$
and for $y \ge 0$, $Y(y):=L^{(1)}_{\tau}(x +y)$. 
It follows from the second Ray--Knight theorem (see \cite{RY}, Theorem (2.3) p.456)  
that $Y$ is equal in law to a squared Bessel process of dimension 0 starting from $\varepsilon^{1/4}$.
Moreover, with this notation, we can write   
\begin{eqnarray}
\label{pepsilon}
{\mathcal P}_\varepsilon = \PP\left[\tau < 1\quad \mbox{and} 
\quad (Y(\varepsilon),\dots,Y(M\varepsilon)) \in \V^*_\varepsilon\right],
\end{eqnarray} 
with
$$\V^*:= (Y(\ell\varepsilon)- L^{(1)}_{1}(x+\ell\varepsilon))_{\ell=1,..., M} + \V,$$
which is an affine space of $\RR^M$, of dimension at most $k-1$ .

\noindent Observe now that even on the event $\{\tau <1\}$, the space $\V^*$ is not independent of $Y$ and $\tau$, since its law depends a priori on $\tau$. 
However, if this was true (and we will see below  how one can reduce the proof to this situation), then ${\mathcal P}_\varepsilon $ 
would be dominated by  
\[ \sup_{V} \, \PP\left[ (Y(\varepsilon),\dots,Y(M\varepsilon)) \in V_{\varepsilon}\right] \, ,
\] 
with the  $\sup$ taken over all affine subspaces $V \subseteq \RR^M$ of dimension at most $k-1$. 
This last term in turn is controlled by the following lemma, whose proof is postponed to the next subsections.
\begin{lem}\label{Besselsanscond}
Let $Y$ be a squared Bessel process of dimension $0$ starting from 
$\varepsilon^{1/4}$. For any $M \ge 1$ and $k \ge 1$, there exists $C > 0$, such that for all $\varepsilon \in (0,1)$,
\begin{eqnarray}
\label{Besselsanscond.eq}
 \sup_V \, \PP\left[(Y(\varepsilon),\dots,Y
(M\varepsilon))
\in V_\varepsilon \right]  \le C\, \varepsilon^{(5M-4(k-1))/8}, 
\end{eqnarray} 
where the $\sup$ is over all affine subspaces $V \subseteq \RR^M$ of dimension at most $k-1$. 
\end{lem}
So at this point we are just led to see how one can solve the problem of the dependence between $\V^*$ and $\tau$.   
To this end, we introduce  the time $\tau'$ spent by $B^{(1)}$ above $x$ before time $\tau$, which by the occupation times formula (see \cite{RY}, Theorem (2.3) p.456) is equal to:
 $$\tau':=\int_0^{\tau}{\bf 1}_{\{B_s^{(1)}\ge x \}}\, ds= \int_0^\infty Y(y)\, dy.$$ 
Moreover, $\tau'$ is also equal in law to the first hitting time of $\varepsilon^{1/4}/2$ by a 
Brownian motion (see the proof of Theorem (2.7) p.243 
in \cite{RY}). In particular 
\begin{equation}
\label{cont.tau'}
\PP(\tau' \le \varepsilon^{3/4}) = \O(\varepsilon^K).
\end{equation}
Next instead of using Lemma \ref{Besselsanscond}, we will need the following refinement:   
\begin{lem}\label{Bessel}
Let $M\ge 1$ be some integer. Let $Y$ be a squared Bessel process of dimension $0$ starting from 
$\varepsilon^{1/4}$. Set
$$A_\varepsilon:=\left\{|Y(M\varepsilon)-\varepsilon^{1/4}|\le \varepsilon^{1/2}\quad
   \mbox{and}\quad \int_0^{M\varepsilon}Y(y)\, dy<\varepsilon\right\}. $$
Then $\PP(A_\varepsilon^c)=\O(\varepsilon^K)$.
Moreover, for any $M \ge 1$ and $k \ge 1$, there exists $C >0$, such that for any affine space $V$ of dimension at most $k-1$,
almost surely for all $\varepsilon \in (0,1)$,
\begin{eqnarray*}
\ind_{ \{\int_0^\infty Y(y)\, dy  \ge \varepsilon^{3/4} \}}\,  \PP\left[(Y(\varepsilon),\dots,Y
(M\varepsilon))
\in V_\varepsilon,\ 
    A_\varepsilon \ \Big|\  \int_0^\infty Y(y)\, dy \right]  \le C\, \varepsilon^{(5M-4(k-1))/8}.
\end{eqnarray*} 
\end{lem}
\noindent We postpone the proof of this lemma to the next subsections, and we conclude now the proof of Lemma \ref{lemeqx2}. 
First it follows from the excursion theory of the Brownian motion that, conditionally to $\tau'$, $Y$ is independent of $\tau$. 
On the other hand, conditionally to $\tau$, $\V^*$ is independent of $\tau'$ and $Y$.
Let $M$ be an integer such that $M \ge (4(k-1)+8K)/5$.
According to \eqref{pepsilon}, \eqref{cont.tau'}
and to the first part of Lemma \ref{Bessel}, we get
\begin{eqnarray*}
{\mathcal P}_\varepsilon
& \le & \PP\left(\tau <1,\  \tau' \ge \varepsilon^{3/4},  
 \  (Y(\varepsilon),\dots,Y(M\varepsilon))  \in \V^*_{\varepsilon},\ A_\varepsilon\right)
+\O(\varepsilon^K) \\
&\le& {\mathbb E}\left[{\mathbf 1}_{\{\tau<1,\tau'\ge \varepsilon^{3/4}
   \}}{\mathbb E}[f(Y,\mathcal V^*)|\tau,\tau']\right]+\O(\varepsilon^K),
\end{eqnarray*}
with 
$$f(y,V):={\mathbf 1}_{\{(y(\varepsilon),...,y(M\varepsilon))\in V_\varepsilon\}\cap
   \{|y(M\varepsilon)-\varepsilon^{1/4}|\le\varepsilon^{1/2}\ \mbox{and}
  \ \int_0^{M\varepsilon}y(s)\, ds<\varepsilon\}}.$$
Now, since on one hand $Y$ and $\mathcal V^*$ are independent conditionally to $(\tau,\tau')$, 
and on the other hand $Y$ and $\tau$ are independent conditionally to $\tau'$, we have
\begin{eqnarray*}
{\mathbb E}[f(Y,\mathcal V^*)|\tau,\tau'] 
    &=&\int {\mathbb E}[f(Y,V)|\tau,\tau']\, d\mathbb P_{\mathcal V^*|(\tau,\tau')}(V)\\
  &=& \int {\mathbb E}[f(Y,V)|\tau']\, d\mathbb P_{\mathcal V^*|(\tau,\tau')}(V).
\end{eqnarray*}
Since moreover, $\mathcal V^*$ and $\tau'$ are independent conditionally to $\tau$,
we get
\begin{equation}
{\mathbb E}[f(Y,\mathcal V^*)|\tau,\tau']  
   =  \int {\mathbb E}[f(Y,V)|\tau']\, d\mathbb P_{\mathcal V^*|\tau}(V).
\end{equation}
Hence, according to our choice of $M$,
\begin{eqnarray*}
{\mathcal P}_\varepsilon
&\le&{\mathbb E}\left[{\mathbf 1}_{\{\tau<1,\tau'\ge \varepsilon^{3/4}
   \}}   \int {\mathbb E}[f(Y,V)|\tau']\, d\mathbb P_{\mathcal V^*|\tau}(V) \right]
+\O(\varepsilon^K)\\
&\le& {\mathbb E}\left[ {\mathbf 1}_{\{\tau<1,\tau'\ge \varepsilon^{3/4}
   \}} \int {\mathbf 1}_{b(V,\varepsilon,\tau')}
    \, d\mathbb P_{\mathcal V^*|\tau}(V) \right]+ \O(\varepsilon^K),
\end{eqnarray*}
with
$$b(V,\varepsilon,t'):=\left\{{\mathbf 1}_{\{\tau'\ge \varepsilon^{3/4}
   \}}   {\mathbb E}[f(Y,V)|\tau'=t']> C\, \varepsilon^{(5M-4(k-1))/8}  \right \}.$$
The second part of Lemma \ref{Bessel} 
insures that, for every affine subspace $V$ of dimension at most $k-1$ of 
${\mathbb R}^M$, we have
$$ {\mathbf 1}_{b(V,\varepsilon,t')}=0 \quad \mbox{for }{\mathbb P}_{\tau'}\mbox{-almost every }
  t'>0.$$
However, since $b(V,\varepsilon,\tau')$ depends a priori on $V$, we cannot conclude directly. 
But it is well known that $\tau'$ admits a positive density function on $(0,+\infty)$ 
(see \eqref{densitetau'} below for an explicit expression).
Therefore, for every $V$,
\begin{equation}\label{Leb}
{\mathbf 1}_{b(V,\varepsilon,t')}=0 \quad \mbox{for Lebesgue almost every }
  t'>0.
\end{equation}
Now it follows from the excursion 
theory that $\tau'$ and $\tau-\tau'$ are independent and identically distributed. 
Therefore $(\tau,\tau')$ admits a continuous density function $h$ on $(0,+\infty)^2$
and we have
\begin{eqnarray}\label{peps}
{\mathcal P}_\varepsilon
&\le& \nonumber \O(\varepsilon^K)+\int_0^1\left(\int_0^{t} \left(\int{\mathbf 1}_{b(V,\varepsilon,t')}
     \, d\mathbb P_{\mathcal V^*|\tau=t}(V)\right)h(t,t')\, dt'\right)\, dt\\
&\le&  \O(\varepsilon^K)+\int_0^1\left(\int\left(\int_0^{t} {\mathbf 1}_{b(V,\varepsilon,t')}
      h(t,t')\, dt'\right)\, d\mathbb P_{\mathcal V^*|\tau=t}(V)\right)\, dt\\
&\le& \nonumber \O(\varepsilon^K),
\end{eqnarray}
the last term of \eqref{peps} being equal to zero according to \eqref{Leb}.
This concludes the proof of Lemma \ref{lemeqx2}. \hfill $\square$

\vspace{0.2cm} 
\noindent It remains now to prove Lemma  \ref{Bessel}. Its proof uses Lemma \ref{Besselsanscond}, so let us start with the proof of the latter. 

\subsection{Proof of Lemma \ref{Besselsanscond}.} 
We first prove the following result:
\begin{lem}
\label{Bessel0}
For every $K>0$ and $M \ge 1$, there exists $C > 0$, such that for all $\varepsilon \in (0,1)$, 
$$
\PP\cro{\exists \ell \in \acc{1, \dots, M} \ :\ |Y(\ell \varepsilon)-\varepsilon^{1/4}| > \varepsilon^{1/2} } 
   \le C \, \varepsilon^K. 
$$
\end{lem}
\begin{proof} 
 Recall that  
$Y$ is solution of the stochastic 
differential equation 
$$Y(y) = \varepsilon^{1/4} + 2\int_0^y \sqrt{Y(u)}\, d\beta_u \quad 
\textrm{for all }y\ge 0,$$
where $\beta$ is a Brownian motion (see \cite{RY} Ch. XI). In particular, 
$Y$ is stochastically dominated by the square of a one-dimensional Brownian motion starting from $\varepsilon^{1/8}$. 
Then it follows that, for some constant $C>0$, whose value may change from line to line, but depending only on $K$ and $M$,   
\begin{eqnarray*}
 && \PP\cro{\exists \ell \in \acc{1, \dots, M} \ :\  |Y(l\varepsilon)-\varepsilon^{1/4}| > \varepsilon^{1/2} }
\\ 
& & \hspace*{1cm} 
\le \PP\cro{\sup_{s \le M\varepsilon} |Y(s)-\varepsilon^{1/4}|> \varepsilon^{1/2} } \\ 
& & \hspace*{1cm} \le   C\, \varepsilon^{-4K}\, \EE\cro{\left(\int_0^{M \varepsilon} Y(u) \, du\right)^{4K} } \mbox{ by the Burkholder-Davis-Gundy inequality, }
\\
& & \hspace*{1cm} \le  C\,  \varepsilon^{-1} \int_0^{M \varepsilon} \EE\cro{Y(u)^{4K}} \, du 
\le  C\, \varepsilon^{-1} \int_0^{M \varepsilon} \EE\cro{ (\varepsilon^{1/8} + B_u)^{8K}} \, du 
\le C\, \varepsilon^{K},  
\end{eqnarray*}
with $B$ some standard Brownian motion. This concludes the proof of the lemma. 
\end{proof} 
\noindent We continue now the proof of Lemma \ref{Besselsanscond}.
Set 
$$\B_{\infty}(\varepsilon^{1/4}, \varepsilon^{1/2}) := \acc{(y_1,\dots,y_M) \in \RR^M\ :\   
\va{y_\ell - \varepsilon^{1/4}} \le  \varepsilon^{1/2} \quad \forall \ell \in \acc{1,\dots,M}}.$$
Lemma \ref{Bessel0} shows that for any $V$ of dimension at most $k-1$,
\begin{eqnarray}
\label{sanscond.1}
 \PP\cro{(Y(\varepsilon),\dots,Y(M\varepsilon))  \in V_{\varepsilon}} 
 \le   \PP\cro{ (Y(\varepsilon),\dots,Y(M\varepsilon))\in \B_{\infty}(\varepsilon^{1/4}, \varepsilon^{1/2}) \cap V_{\varepsilon}}+ C\, \varepsilon^K \, .
\end{eqnarray}
Next observe that   
 $B_\infty(\varepsilon^{1/4}, \varepsilon^{1/2})\cap V_\varepsilon$, can be covered by $\O(\varepsilon^{-(k-1)/2})$ balls of radius $\varepsilon$. 
It follows that  
\begin{eqnarray}
\label{sanscond.2}
\nonumber &&  \PP\cro{ (Y(\varepsilon),\dots,Y(M\varepsilon))\in \B_{\infty}(\varepsilon^{1/4}, \varepsilon^{1/2}) \cap V_{\varepsilon}} \\ 
&& \hspace{1cm} \le C\, \varepsilon^{-(k-1)/2}\, \sup_{x \in \B_{\infty}(\varepsilon^{1/4}, \varepsilon^{1/2})}
\PP\cro{ (Y(\varepsilon),\dots,Y(M\varepsilon))\in \B_{\infty}(x,\varepsilon)} \, .
\end{eqnarray}
Now for $y>0$, denote by $Y_y$ a squared Bessel process with dimension $0$ starting from $y$. An explicit expression of its semigroup is 
given just after Corollary (1.4) p.441 in \cite{RY}. 
In particular when $y>\varepsilon^{1/4}/2$, the law of $Y_y(\varepsilon)$ is the sum of a Dirac mass at $0$ with some negligible weight 
and of a measure with density 
$$z\mapsto q_\varepsilon(y,z) := (2\varepsilon)^{-1} \sqrt{\frac y z} 
 \exp\pare{- \frac{y+z}{2\varepsilon}} I_1\left(\frac{\sqrt{yz}}{\varepsilon}\right),$$
where $I_1$ is the modified Bessel function of index $1$. 
Moreover it is known (see (5.10.22) or (5.11.10) in \cite{L}), that 
$I_1(z) = \O(e^z/\sqrt z)$, as $z\to \infty$. 
Thus
$$\sup_{\va{y-\varepsilon^{1/4}} \le \varepsilon^{1/2} } \, 
\,  \sup_{\va{ z-\varepsilon^{1/4}} \le \varepsilon^{1/2}} 
q_{\varepsilon}(y,z) = \O(\varepsilon^{-3/8}).$$
It follows that 
$$\sup_{\va{x-\varepsilon^{1/4}} \le \varepsilon^{1/2} } \, 
\,  \sup_{\va{y-\varepsilon^{1/4}} \le \varepsilon^{1/2}} 
\PP\left[ |Y_y(\varepsilon)-x| \le \varepsilon\right]  = \O(\varepsilon^{5/8}).$$
Then by using the Markov property and Lemma \ref{Bessel0}, we get by induction 
\begin{eqnarray}
\label{sanscond.3}
\sup_{x \in \B_{\infty}(\varepsilon^{1/4}, \varepsilon^{1/2})}
\PP\cro{ (Y(\varepsilon),\dots,Y(M\varepsilon))\in \B_{\infty}(x,\varepsilon)} \le C\, \varepsilon^{5M/8}.
\end{eqnarray}
Since all the constants in our estimates are uniform in $V$, 
Lemma \ref{Besselsanscond} follows from \eqref{sanscond.1}, \eqref{sanscond.2} and \eqref{sanscond.3}. \hfill $\square$

\subsection{Proof of Lemma \ref{Bessel}.\/}
Let $K>0$ be given. 
Lemma \ref{Bessel0} shows in particular that   
$$
\PP\left[|Y(M \varepsilon)-\varepsilon^{1/4}|
> \varepsilon^{1/2} \right] = \O(\varepsilon^{K}).
$$
Next, recall that $Y$ is stochastically dominated by the square of a one-dimensional 
Brownian motion starting from $\varepsilon^{1/8 }$. It follows that  
$$\PP\left(\int_0^{M\varepsilon} Y(y)\, dy \ge \varepsilon\right) = \O(\varepsilon^K),$$
and this already proves the first part of the lemma.

\noindent It remains to prove the second part. We deduce  it from Lemma \ref{Besselsanscond}.
To simplify notation,  from now on we will denote  the integral of $Y$ on $[0,\infty)$  by $\int_0^\infty Y$. 
Likewise $\int_0^{M \varepsilon} Y$ and $\int_{M \varepsilon}^\infty Y$ 
will have analogous meanings. For any affine subspace $V\subseteq \RR^M$ of dimension at most 
$k-1$, set 
$$A'_\varepsilon(V) := \left\{ (Y(\varepsilon),\dots,Y(M\varepsilon)) \in V_\varepsilon\right\}
   \cap A_\varepsilon.$$
Then for any nonnegative bounded measurable
function $\phi$  
supported on $[\varepsilon^{3/4},\infty)$, we can write 
\begin{eqnarray*}
\EE \cro{\phi\pare{\int_0^\infty Y} \PP\cro{A'_\varepsilon(V) \, \Big| \, \int_0^\infty Y}}
& = & \EE \cro{\phi\pare{\int_0^\infty Y} ,\  A'_\varepsilon(V)}
\\
& = & \EE \cro{\phi \pare{\int_0^{M \varepsilon} \ Y + \int_{M \varepsilon}^\infty Y},\ A'_{\varepsilon}(V)}.
\end{eqnarray*}
Now we recall that if $Y_y$ denotes a squared Bessel process of dimension $0$ starting from some $y>0$, then 
$\int_0^\infty Y_y$ is equal in law to the first hitting time of $y/2$ by some Brownian motion, and thus has density given by 
\begin{equation}
\label{densitetau'}
f_y(t) := \frac y2 (2\pi t^3)^{-1/2} \exp (-(y/2)^2/2t)\quad \mbox{for all }t>0\mbox{ and }y> 0, 
\end{equation} 
see for instance \cite{RY} p.107. In particular 
$$\sup_{t \ge \varepsilon^{3/4}}\, \sup_{t' \le \varepsilon}\, \sup_{\va{y-\varepsilon^{1/4}} \le \varepsilon^{1/2}}\, \frac{f_{y}(t-t')}{f_{\varepsilon^{1/4}}(t)}<\infty.$$
Then by using the Markov property and Lemma \ref{Besselsanscond}, we get
\begin{eqnarray*}
 \EE \cro{\phi\pare{\int_0^\infty Y} \PP\cro{A'_\varepsilon(V) \, \Big| \, \int_0^\infty Y}}
&=&  \EE \cro{ \int_{\varepsilon^{3/4}}^\infty \phi(t)  \, f_{Y(M\varepsilon)}\pare{t-\int_0^{M\varepsilon} Y}
\, dt,\ A'_{\varepsilon}(V)} \\
& \le & C \,  \PP \cro{A'_{\varepsilon}(V) }  
\EE\cro{\phi \pare{\int_0^{\infty} Y}}
\\
& \le & C\, \varepsilon^{(5M-4(k-1))/8}\, \EE\cro{\phi \pare{\int_0^{\infty} Y}} \, .
\end{eqnarray*}
Since this holds for any $\phi$, this proves the second part of Lemma \ref{Bessel}, as wanted.    
\hfill $\square$

%%%%%%%%%%%%%%%%%%%%%%%%%%%%%%%%%%%%%%%%%%%%%%%%%%%%%%%%%%%%%%%%%%%%%%%%%%%%%%%%%%%%%%%%%%%%%%%%%%%%%%%%%%%%%%%%%%%%%%%%%%%%%%%%%%
%%%%%%%%%%%%%%%%%%%%%%%%%%%%%%%%%%%%%%%%%%%%%%%%%%%%%%%%%%%%%%%%%%%%%%%%%%%%%%%%%%%%%%%%%%%%%%%%%%%%%%%%%%%%%%%%%%%%%%%%%%%%%%%%%%

\section{Proof of Theorem \ref{thmLT} and Corollary \ref{cor2}} 
\label{secLT}
We start with the proof of Theorem \ref{thmLT}. 
We follow the general strategy which is used in the case of the Brownian motion, as for instance in Le Gall's course \cite[Chapter 2]{LG}.

\noindent Consider the regularizing function 
$$p_\varepsilon(y):= \frac 1{\sqrt{2\pi \varepsilon}} \exp\left(-\frac {y^2}{2\varepsilon}\right) \quad \varepsilon >0 \quad y\in \RR,$$
and recall that by Fourier inversion
$$p_\varepsilon(y)= \frac 1{2\pi} \int_\RR \exp(iy\xi - \frac 12 \varepsilon |\xi|^2)\, d\xi.$$
Define then for all $\varepsilon \in (0,1]$, $t>0$, and $x\in \RR$, 
$$\L(\varepsilon,t,x) := \int_0^t p_\varepsilon (\Delta_s-x)\, ds. $$ 
As explained in \cite{LG}, it suffices to control the three terms: 
$$\EE\left[\left(\L(\varepsilon,t,x)-\L(\varepsilon,t,x')\right)^{2p}\right],\  
\EE\left[\left(\L(\varepsilon,t,x)-\L(\varepsilon',t,x)\right)^{2p}\right],\ \EE\left[\left(\L(\varepsilon,t,x)-\L(\varepsilon,t',x)\right)^{2p}\right].$$
For the first term, some elementary computation shows that  
\begin{eqnarray} 
\label{1terme}
\nonumber && \EE\left[\left(\L(\varepsilon,t,x)-\L(\varepsilon,t,x')\right)^{2p}\right] \\ 
&\le &  c_p \int_{\RR^{2p}} d\xi_1\dots d\xi_{2p} \prod_{j=1}^{2p} \va{e^{-ix\xi_j}-e^{-ix'\xi_j}}
 \, \int_{\Xi_p} ds_1\dots ds_{2p} \, 
\va{\EE\left[e^{i\sum \xi_j \Delta_{s_j}} \right]}  \, ,
\end{eqnarray}
with $c_p$ some positive constant (whose value may change in the following lines) and $\Xi_p:=\{s_1\le \dots \le s_{2p}\le t\}$. 
We use next that for any $\gamma \in (0,1]$, and any $y$, $y'\in \RR$, 
$$\left|e^{iy}-e^{iy'}\right|\le c\, |y-y'|^\gamma,$$
for some constant $c>0$. Moreover, if $\eta_j=\xi_j+ \dots +\xi_{2p}$, and $t_j=s_{j}-s_{j-1}$, for all $j\ge 1$ (with the convention $s_{0}=0$), then 
$$ \EE\left[e^{i\sum \xi_j \Delta_{s_j}} \right] = \EE\left[ e^{-\frac 12 \sum_{i,j} \eta_i\eta_j \langle L^{(i)}_{t_i},L^{(j)}_{t_j}\rangle}\right]
=\EE\left[ e^{-\frac 12 \langle \widetilde M_{t_1,\dots,t_{2p}} \eta,\eta \rangle}\right],$$
with $\eta=(\eta_1,\dots,\eta_{2p})$.  
Therefore a change of variables in \eqref{1terme} gives 
\begin{eqnarray*} 
&& \EE\left[\left(\L(\varepsilon,t,x)-\L(\varepsilon,t,x')\right)^{2p}\right] \\ 
&\le &  c_p\, |x-x'|^{2\gamma p}\, \int_{\RR^{2p}} d\eta \, \prod_{j=1}^{2p} |\eta_{j+1}-\eta_{j}|^\gamma\,
\left(\int_{[0,t]^{2p}}   dt_1\dots dt_{2p} \, \, 
\EE\left[e^{-\frac 12 \langle \widetilde M_{t_1,\dots,t_{2p}} \eta,\eta\rangle}
\right]\,\right), 
\end{eqnarray*}
with the convention $\eta_{2p+1}=0$. Now we make another change of variables: $(\eta_1,\dots,\eta_{2p})\to (\eta_1/t_1^{3/4},\dots,\eta_{2p}/t_{2p}^{3/4})$. 
Then we fix some $T>0$, and by using also that for all $j$, and $t\le T$, 
$$|t_{j+1}^{-3/4}\eta_{j+1}-t_{j}^{-3/4}\eta_{j}|\le c\, \max(t_{j+1}^{-3/4},t_{j}^{-3/4})\, |\eta| 
	\le c\, t_{j+1}^{-3/4} \, t_{j}^{-3/4} T^{3/4} \, |\eta| 	,$$
for some constant $c>0$, we get for all $t\le T$,  
\begin{eqnarray*} 
&&\EE\left[\left(\L(\varepsilon,t,x)-\L(\varepsilon,t,x')\right)^{2p}\right] \\ 
&\le &  c_{p,T} \, |x-x'|^{2\gamma p}\, \int_{[0,t]^{2p}}  dt_1\dots dt_{2p}
\left(\prod_{j=1}^{2p} t_j^{-3/4(1+2\gamma )} \right)  \int_{\RR^{2p}} \, d\eta\, 
\EE\left[e^{-\frac 12 \langle \overline M_{t_1,\dots,t_{2p}} \eta,\eta\rangle}
\right]\,  |\eta|^{2\gamma p}\, ,
\end{eqnarray*} 
for some constant $c_{p,T}>0$. 
Now Proposition \ref{proplambda} shows that all moments of 
$1/\overline \lambda_{t_1,\dots,t_{2p}}$ are bounded by positive constants, uniformly in $(t_1,\dots,t_{2p})$. 
Therefore by using that for all $\eta$, 
$$\langle \overline M_{t_1,\dots,t_{2p}}\eta,\eta\rangle\, \ge \overline \lambda_{t_1,\dots,t_{2p}}\, |\eta|^2,$$ 
and the change of variables $\eta\to \eta/(\overline \lambda_{t_1,\dots,t_{2p}})^{1/2}$, we get for all $\gamma<1/6$, 
$$\EE\left[\left(\L(\varepsilon,t,x)-\L(\varepsilon,t,x')\right)^{2p}\right] 
\le  c_{p,T} \, |x-x'|^{2\gamma p},$$
for all $t\le T$.

\noindent A similar computation yields to an analogous estimate for the second term, except that 
this time we need to choose $\gamma<1/12$: for all $p \ge 1$, $T>0$, and $\gamma<1/12$, there exists some constant $c'_{p,T} >0$, 
such that for all $x \in \RR$, all $\varepsilon, \varepsilon'   >0$, and all $t\le T$, 
$$\EE \cro{\pare{ \L(\varepsilon,t,x) - \L(\varepsilon',t,x)}^{2p}}
 \le c'_{p,T} \va{\varepsilon-\varepsilon'}^{2p\gamma} \, .
 $$
Now the estimate of the last term is easier. After some calculation and by using Theorem \ref{theoMk}, we get for $t<t'$, 
\begin{eqnarray*} 
\EE\left[\left(\L(\varepsilon,t,x)-\L(\varepsilon,t',x)\right)^{2p}\right] 
\le   c_p\, \int_{t\le s_1\le \dots\le s_{2p}\le t'}\, \frac{ds_1\dots ds_{2p}}{\prod(s_j- s_{j-1})^{3/4}},
\end{eqnarray*}
which shows that 
$$\EE\left[\left(\L(\varepsilon,t,x)-\L(\varepsilon,t',x)\right)^{2p}\right] \le c_p\, |t'-t|^{p/2}.$$
Then Part (i) and (ii) in Theorem \ref{thmLT} follow from Kolmogorov's criterion (see \cite{LG} for details). 
For (iii), first observe that (ii) implies that a.s. for any $t>0$ and $x\in \RR$, 
$$\L_t(x) = \lim_{\varepsilon \to 0} \, \frac 1{2\varepsilon} \int_0^t {\bf 1}_{\{\Delta_s \in  [x-\varepsilon,x+\varepsilon] \}}\, ds.$$
Then (iii) immediately follows from this equation and the property of self similarity of $\Delta$. 
For (iv), we can observe that by using the above computations and the dominated convergence theorem, we get 
$$\EE[\L_t(x)^k] \, =\, \lim_{\varepsilon \to 0} \, \EE[\L(\varepsilon,t,x)^k].$$
Part (iv) follows. Part (v) is immediate, and was already observed in \cite{YX}.

\noindent Concerning Corollary \ref{cor2}, the upper bound was already proved in \cite{YX} and \cite{K} and was considered there as the easiest part. 
So we only care about the lower bound here. 
For this we can use Frostman's Lemma together with Theorem \ref{thmLT}, which directly proves the result (see \cite{LG} for instance).

%%%%%%%%%%%%%%%%%%%%%%%%%%%%%%%%%%%%%%%%%%%%%%%%%%%%%%%%%%%%%%%%%%%%%%%%%%%%%%%%%%%%%%%%%%%%%%%%%%%%%%%%%%%%%%%%%%%%%%%%%%%%%%%%%%
%%%%%%%%%%%%%%%%%%%%%%%%%%%%%%%%%%%%%%%%%%%%%%%%%%%%%%%%%%%%%%%%%%%%%%%%%%%%%%%%%%%%%%%%%%%%%%%%%%%%%%%%%%%%%%%%%%%%%%%%%%%%%%%%%%
\section{Proof of Theorem \ref{thmTLL}} 
\label{secTLL}
In most of this section, $t_1,\dots,t_k$, are fixed positive reals. 
Moreover, by convention a function $f(n_1,\dots,n_k)$ is said to be a $o_k(g(n))$, for some function $g$,  
if it converges to $0$ after multiplication by $1/g(n)$, 
when $n\to \infty$ and $n_i/n \to t_i$ for all 
$i\ge 1$. Analogous convention is used for the notation $\O_k(g(n))$.

\noindent Recall that $(S_m,m\ge 0)$ denotes the random walk. 
For every $i=1,...,k$, let $(N^{(i)}_m(x),1\le m \le n_i, x\in \ZZ)$, 
be the local time process of $\left(S^{(i)}_m:=S_{n_1+...+n_{i-1}+m},0\le m\le n_i-1\right)$. 
In other words, 
\begin{eqnarray*}
N_m ^{(i)}(x)&:=&\#\{k=0,...,m-1\ :\ S_{n_1+...+n_{i-1}+k}=x\}\\
&=&N_{n_1+...+n_{i-1}+m}(x)- N_{n_1+...+n_{i-1}}(x),
\end{eqnarray*}
for all $i\le k$. Set also 
$$D_{n_1,\dots,n_k}:= \det\left(\langle N^{(i)}_{n_i},N^{(j)}
_{n_j}\rangle\right)_{1\le i,j\le k},$$
where here $\langle\cdot,\cdot\rangle$ denotes the usual 
scalar product on $\ell_2(\ZZ)$.

\subsection{Inverse Fourier transform and a periodicity issue.}
The first step in local limit theorems is often the use of Fourier inverse transform. 
This is essentially the content of the next lemma. Before stating it, let us introduce 
some new notation. 
Recall that $\varphi_\xi$ denotes the characteristic function of $\xi_0$. Let now 
$\varphi_{n_1,\dots,n_k}$ be the characteristic function of 
$(Z_{n_1+\dots+n_i}-Z_{n_1+\dots+n_{i-1}})_{i=1,...,k}$. 
Since $(\xi_y)_{y\in\mathbb Z}$ is a sequence of i.i.d. random
variables, which is independent of $S$, we have
for all $(\theta_1,\dots,\theta_k) \in \RR^k$,
\begin{eqnarray}
\label{phini}
\nonumber \varphi_{n_1,\dots,n_k}(\theta_1,\dots,\theta_k)&:=& 
\EE\left[\prod_{y\in \ZZ}\varphi_\xi\left( \sum_{j=1}^k\theta_j(N_{n_1+...+n_j}(y)
- N_{n_1+...+n_{j-1}}(y))\right)\right]\\
&=& \EE \left[\prod_{y\in \ZZ}
  \varphi_{\xi}\left( \sum_{j=1}^k\theta_jN_{n_j}^{(j)}(y)\right)\right].
\end{eqnarray}
We can now state the announced lemma.
\begin{lem}
\label{formule1}
If $n_i\in d_0\ZZ$ for all $i\le k$, then 
$$
\PP\left(Z_{n_1}=\dots=Z_{n_1+\dots+n_k}=0\right) =  
\left(\frac {d}{2\pi}\right)^k 
   \int_{[-\frac{\pi}d,\frac\pi d]^k} \varphi_{n_1,\dots,n_k}(\theta_1,\dots,\theta_k)
   \, d\theta_1\dots d\theta_k.
$$
Otherwise $ \PP(Z_{n_1}=\dots=Z_{n_1+\dots+n_k}=0)=0$. 
\end{lem}
\begin{proof}
Since $Z$ is $\ZZ$-valued, we immediately get 
\begin{eqnarray*}
{\mathbb P}(Z_{n_1}=\dots=Z_{n_1+\dots+n_k}=0) 
&=& \PP (Z_{n_1}= \dots = Z_{n_1+\dots+n_k} - Z_{n_1+\dots+n_{k-1}}=0)\\
&=&\frac 1{(2\pi)^k}
   \int_{[-\pi, \pi]^k}  \varphi_{n_1,\dots,n_k}(\theta_1,\dots,\theta_k)\, d\theta_1\dots d\theta_k.
\end{eqnarray*}
Notice now that $e^{{2i\pi\xi_0/d}}=\varphi_\xi(2\pi/d)$ almost surely and that
$\varphi_\xi(2\pi/d)^d=1$.
Hence, for any integer $m\ge 0$ and any $u\in \RR$,
$$\varphi_\xi\left(2m\pi/d+u\right)
=\varphi_\xi\left(2\pi/d\right)^m\varphi_\xi(u).$$
We deduce that, for every $(l_1,\dots,l_k)\in \ZZ^k$, we have
\begin{eqnarray*}
& & \varphi_{n_1,\dots,n_k}\left(\theta_1+\frac{2l_1\pi}{d},\dots,\theta_k+\frac{2l_k\pi}{d}\right)
=\EE\left[\prod_{y\in \ZZ}   \varphi_{\xi}\left( \sum_{j=1}^k
\left( \theta_j +\frac{2 l_j\pi} d\right)N_{n_j}^{(j)}(y)\right)\right] \\
&=& \EE\left[\prod_{y\in \ZZ}
  \varphi_\xi(2\pi/d)^{\sum_{j=1}^kl_jN_{n_j}^{(j)}(y)}
  \varphi_\xi\left (\sum_{j=1}^k\theta_jN_{n_j}^{(j)}(y)\right)\right]\\
& =& \varphi_\xi(2\pi/d)^{\sum_{j=1}^kl_jn_j}\ \varphi_{n_1,\dots,n_k}(\theta_1,\dots,\theta_k),
\end{eqnarray*}
since $\sum_y N^{(j)}_{n_j}(y)=n_j$.
But, if $n_j\in d_0\ZZ$ for all $j\le k$, 
then $\varphi_\xi(2\pi/d)^{\sum_{j=1}^kn_jl_j}=1$, for all $(l_1,\dots,l_k)\in \ZZ^k$, 
and the result follows with a change of variables. 
If not, let $j$ be such that $n_j\notin d_0\ZZ$. Then $\varphi_\xi(2\pi/d)^{n_j}$ 
is a nontrivial $d$-th root of unity
and we can write
\begin{eqnarray*} 
\PP(Z_{n_1}=&\dots &=Z_{n_1+\dots+n_k}=0)\ =\   \frac{1}{(2\pi)^k} 
\left(\sum_{l_j=0}^{d-1}\varphi_\xi(2\pi/d)^{n_jl_j}\right) \\ 
 & &\times \int_{[-\pi,\pi]^{k-1}} 
  \left[\int_{-[\frac \pi d ,\frac \pi d]}\varphi_{n_1,\dots,n_k}(\theta_1,\dots,\theta_k)
  \, d\theta_j\right]
  \, d\theta_1\dots d\theta_{j-1}d\theta_{j+1}\dots d\theta_k \\
&=& 0. 
\end{eqnarray*}
This concludes the proof of the lemma. 
\end{proof}

\subsection{A typical behaviour for random walks.}
\label{sec:omega_n} 
We want to argue that typically the simple random walk visits roughly $\sqrt n$ sites before time $n$; 
spends time of order at most $\sqrt n$ on each of them, and that its local time process is H\"older 
continuous of order $1/2$, with a H\"older constant in $\O(n^{1/4})$.   
This is true with high probability if we allow some correction of order $n^\gamma$, with $\gamma>0$. 
This is the content of the next lemma, which can be proved as Lemma 6 in \cite{BFFN} and is standard. 
Set for all $i\le k$, 
$$N_i^*:= \sup_y N_{n_i}^{(i)}(y)\quad \mbox{and}\quad  
R_i := \# \{y\ :\ N_{n_i}^{(i)}(y)>0\}.
$$
\begin{lem}\label{lem:omega_n}
For every $n\ge 1$ and $\gamma > 0$, set 
$\Omega_{n_1,...,n_k}:=\Omega_{n_1,...,n_k}^{(1)}\cap 
        \Omega_{n_1,...,n_k}^{(2)}$, where
$$\Omega_{n_1,...,n_k}^{(1)}:= \left\{R_i 
    \le n_i^{\frac 1 2+\gamma}\quad \forall i\le k\right\},$$
and 
$$\Omega_{n_1,...,n_k}^{(2)}:= \left\{\sup_{y\neq z} 
 \frac{|N_{n_i}^{(i)}(y)-N_{n_i}^{(i)}(z)|}{|y-z|^{1/2}} \le n_i^{\frac 14 +\gamma}\quad \forall i\le k\right\}.$$
Then, for every $p$, $\PP(\Omega_{n_1,\dots,n_k}^c) = o(\min_i n_i^{-p})$.
\end{lem}
\noindent Note that on $\Omega_{n_1,...,n_k}$, for every $i$, we have 
$$N_i^* \le n_i^{\frac 1 2+\gamma} \quad \textrm{and}\quad V_{n_i}^{(i)}:=\sum_y(N_{n_i}^{(i)}(y))^2\le
     n_i^{\frac 32+3\gamma}. $$

%%%%%%%%%%%%%%%%%%%%%%%%%%%%%%%%%%%%%%%%%%%%%%%%%%%%%%%%%%%%%%%%%%%%%%%%%%%%%%%%%%%%%%%%%%%%%%%%%%%%%%%%%%%%%%%%%%%%%%%%%%%%%%%%%%%%%%%%
%%%%%%%%%%%%%%%%%%%%%%%%%%%%%%%%%%%%%%%%%%%%%%%%%%%%%%%%%%%%%%%%%%%%%%%%%%%%%%%%%%%%%%%%%%%%%%%%%%%%%%%%%%%%%%%%%%%%%%%%%%%%%%%%%%%%%%%%%
%
%
\subsection{Scheme of the proof.}\label{sec:scheme}
We follow roughly the same lines as for the proof of Theorem 1 in \cite{BFFN}. 
However the situation is more complicated here, since we consider multiple times in a non-markovian context.
Moreover, we want upper bounds which are uniform in $n_1,\dots,n_k$, and this also requires some additional care.

\noindent First we have to see that the main contribution in the estimate comes from the integral 
near the origin. Recall in particular the notation from \eqref{phini}. 
\begin{prop}\label{lem:equivalent} 
Let $\eta\in (0,1/8)$ be given. Then, for every $t_1,\dots,t_k\in(0,1)$, we have
\[
  \int_{U(\eta)}
  \varphi_{n_1,\dots,n_k}(\theta_1,\dots,\theta_k) \, d\theta_1\dots d\theta_k 
 = \left(\frac{\sqrt {2\pi}}{\sigma}\right)^k \C_{t_1,\dots,t_k}\, n^{-3k/4} + o_k(n^{-3k/4}),
\]
where $U(\eta):=\{ |\theta_i| \le n_i^{-\frac 12-\eta}\ \forall i\le k\}$. Moreover, for every $\theta\in (0,1)$,
\[
\sup_{n\ge 1}\, \sup_{n^{\theta}\le n_1,\dots,n_k\le n}\, \left(\prod_{i=1}^kn_i^{\frac 34}\right)\, 
 \left| \int_{U(\eta)}
  \varphi_{n_1,\dots,n_k}(\theta_1,\dots,\theta_k) \, d\theta_1\dots d\theta_k 
  \right| <\infty.
\]
\end{prop}
\noindent The next two propositions show that the rest of the integral is negligible. 

\begin{prop}\label{sec:step1}
Let $\eta \in (0,1/8)$ be given. Then, for every $t_1,\dots,t_k\in(0,1)$, we have
$$\int_{V(\eta)} 
  |\varphi_{n_1,\dots,n_k}(\theta_1,\dots,\theta_k)|
  \, d\theta_1\dots d\theta_k = o_k(n^{-3k/4})$$
where $V(\eta):=\{|\theta_i| \le n_i^{-\frac 1 2 +\eta}\ \forall i\le k\} \cap \{ 
\exists j\ :\ |\theta_j|\ge n_j^{-\frac 1 2 - \eta} \}$. Moreover, for every $\theta \in (0,1)$, 
$$
  \sup_{n\ge 1}\, \sup_{n^{\theta}\le n_1,\dots,n_k\le n}\, 
  \left(\prod_{i=1}^k n_i^{\frac 34}
\right)\, \int_{V(\eta)} 
 |\varphi_{n_1,\dots,n_k}(\theta_1,\dots,\theta_k)|
 \, d\theta_1\dots d\theta_k <\infty.$$
\end{prop}

\begin{prop}
\label{proppointes}
Let $\eta\in(0,1/2)$ and $\theta\in (0,1)$ be given. 
Then there exists $c>0$ such that 
$$\sup_{n\ge 1}\, \sup_{n^{\theta}\le n_1,\dots,n_k\le n}\, 
 \int_{\{\exists i\ :\ |\theta_i| > n_i^{-\frac 1 2 +\eta}\}}
  |\varphi_{n_1,\dots,n_k}(\theta_1,\dots,\theta_k)|\, d\theta_1\dots d\theta_k 
=o(e^{-n^c}).$$
\end{prop}
This last proposition can be proved by using exactly the same argument as in the proof of 
Proposition 10 in \cite{BFFN}. The only difference
is that, if say $|\theta_i|>n_i^{-1/2+\eta}$, 
then after having defined peaks for $S^{(i)}$,  
we need to work also conditionally to all 
$N^{(j)}_{n_j}$, for $j\ne i$. But this does not change anything to the proof.  
Since it would be fastidious to reproduce the argument, we will not prove this proposition 
here and we refer the reader to \cite{BFFN} for details.

\noindent Note that Theorem \ref{thmTLL} readily follows from these propositions and Lemma \ref{formule1}. 

\subsection{Proof of Proposition \ref{lem:equivalent}.}
We will use Borodin's result \cite{Bor} on approximations of 
Brownian local time by random walks local time. 
He proved in particular (see Remark 1.3 in \cite{Bor}) that under some moment 
condition on the random walk, and on a suitable probability space, 
for all $T>0$ and all $\gamma >0$, 
there exist constants $C>0$ and $\delta>0$, such that for all $n\ge 1$,
$$\PP\left(E_n^c\right)\, \le\, C\, n^{-1-\delta},$$
with
$$E_n:=\left\{ \sup_{(t,x)\in [0,T]\times \RR} |N_{[nt]}([\sqrt nx]) - 
\sqrt n L_t(x)| \le Cn^{\frac 14} \ln n,\ \left|B_1-\frac{S_n}{\sqrt{n}}
\right|\le n^{-\frac 14+\gamma}\right\},$$
where $N$ and $L$ are the local time processes, respectively of the 
random walk $S$ and of the Brownian motion $B$. But a careful look at his proof shows actually 
that if the random 
walk increments have finite moments of any order, 
then the above holds for any $\delta>0$ (see Formulas (3.8) 
and (3.9) and Lemma 3.2).

\noindent By using now this result, Lemma \ref{lem:omega_n} and the Markov property of the random walk and of Brownian motion, we deduce the following: 
\begin{lem}
Let $\gamma\in(0,1/4)$ and $k\ge 1$ be given. 
Then for every $n\ge 1$ and every $0\le n_1,...,n_k\le n$, 
it is possible to construct the Brownian motion and the random walk on a suitable probability space, 
such that for all $p>0$, 
$${\mathbb P}\left(F_{n,n_1,\dots,n_k}^c\right)\, =\, \O\left((\min_i n_i)^{-p}\right),$$
where $F_{n,n_1,\dots,n_k}=F_1(n,n_1,\dots,n_k)\cap\dots \cap F_4(n,n_1,\dots,n_k)$, and (with $t_i=n_i/n$ for $i\le k$), 
\begin{eqnarray*} 
F_1(n,n_1,\dots,n_k)&:= & \left\{\sup_{x\in\mathbb R} 
     \left|N_{n_i}^{(i)}(\sqrt{n}x)-\sqrt{n}L_{t_i}^{(i)}(x)\right|
    \le n_i^{\frac 14+\gamma} \quad \forall i\le k\right\},\\
F_2(n,n_1,\dots,n_k)&:= & \left\{\sup\{|x-S_0^{(i)}|\, :\, N_{n_i}^{(i)}(x)\ne 0\}\le t_i^{1/2}n^{\frac 12+\gamma}\quad \forall i\le k\right\},\\
F_3(n,n_1,\dots,n_k)&:=& \left\{ \sup\{|x-B_0^{(i)}|\, :\, L_{t_i}^{(i)}(x)\ne 0\}\le t_i^{1/2}n^{\gamma} \quad \forall i\le k \right\},\\
F_4(n,n_1,\dots,n_k)&:=& \left\{ \sup_kN_{n_i}^{(i)}(k)\le t_i^{1/2} n^{\frac 12 +\gamma}\ \ \ \mbox{and}\ \ \ 
  \sup_x L_{t_i}^{(i)}(x)\le t_i^{1/2}n^\gamma \quad \forall i\le k\right\}.
\end{eqnarray*}
\end{lem}
\noindent The proof of this result is elementary and left to the reader. Define now for all $\varepsilon>0$, the set 
\begin{eqnarray}
 \label{omegatilde}
\widetilde \Omega_{n_1,\dots,n_k}(\varepsilon):=  \left\{\left(\prod_{i=1}^k n_i^{-\frac 32}\right)D_{n_1,...,n_k}\ge\varepsilon\right\}. 
\end{eqnarray}
We then obtain the following: 
\begin{lem}\label{lem:borne}
Let $\theta\in(0,1)$ and $\theta_0\in (0,\theta/4)$ be given.
Then for every $L>0$, we have
$$\sup_{n\ge 1}\, \sup_{n^{\theta}\le n_1,...,n_k\le n} \,
      \sup_{\varepsilon\geq  n^{-\theta_0}}\, \varepsilon^{-L}\, {\mathbb P}\left(\widetilde \Omega_{n_1,\dots,n_k}^c(\varepsilon)\right) <\infty,$$
and for every $p>0$, 
$$\sup_{n\ge 1}\, \sup_{n^{\theta}\le n_1,...,n_k\le n}\, 
      {\mathbb E}\left[
     \left(\prod_{i=1}^k n_i^{\frac {3p}2}\right)\, D_{n_1,...,n_k}^{-p},\ \widetilde \Omega_{n_1,\dots,n_k}(n^{-\theta_0}) 
     \right]<\infty.$$
\end{lem}
\begin{proof}
Let $\gamma>0$ be such that $\theta_0<(\theta/4)-3\gamma k$ and let $L>0$ be fixed. 
Thanks to the previous lemma we can assume that the Brownian motion $B$ and the random walk $S$ are constructed on a space, where  
$$\PP(F_{n,n_1,\dots,n_k}^c)=\O(n^{-p}),$$
for all $p>0$. Now for all $i,j$, set 
$$A_{i,j}^{(n)}:=(n_in_j)^{-3/4}\sum_yN_{n_i}^{(i)}(y)N_{n_j}^{(j)}(y)\ \ \ \mbox{and}\ \ \ 
     \A_{i,j}:=(t_it_j)^{-3/4}\, \int_{\mathbb R}L_{t_i}^{(i)}(x)L_{t_j}^{(j)}(x)\, dx,$$
with $t_i=n_i/n$ and $t_j=n_j/n$.
First, we rewrite $A_{i,j}^{(n)}$ as follows
$$ A_{i,j}^{(n)}=(t_it_j)^{-\frac 34} \int_{\mathbb R} 
     \frac{N_{n_i}^{(i)}(\floor{\sqrt{n}x})}{\sqrt{n}}
    \frac{N_{n_j}^{(j)}(\floor{\sqrt{n}x})}{\sqrt{n}}\, dx. $$
Observe next that, on $F_{n,n_1,\dots,n_k}$, for all $i,j$ and $n_i,n_j \le n$, we have
\begin{equation}\label{aii}
A_{i,i}^{(n)}\le n^{3\gamma},\qquad  \A_{i,i}\le n^{3\gamma},
\end{equation}
and
$$t_i^{-\frac 32}
  \int_{\mathbb R} \left|\frac{N_{n_i}^{(i)}(\floor{\sqrt{n}x})}{\sqrt{n}}
   - L_{t_i}^{(i)}(x)\right|^2\, dx
   \le t_i^{-\frac 32}2 t_i^{\frac 12}n^\gamma t_i^{\frac 12}n^{-\frac 12+2\gamma}
    \le 2t_i^{-\frac 12}n^{-\frac 12+3\gamma} .$$
Hence, with the use of the Cauchy-Schwartz inequality, we get
\begin{equation}\label{aij}
A_{i,j}^{(n)}\le n^{3\gamma},\quad  \A_{i,j}\le n^{3\gamma}
\quad \mbox{and}\quad
\left| A_{i,j}^{(n)}-\A_{i,j}\right|\le 2\sqrt{2} t_i^{-\frac 14} n^{-\frac 14+3\gamma}
  \le 4 n^{-\frac \theta 4+3\gamma}.
\end{equation}
We use next that 
$$\left(\prod_{i=1}^k n_i^{-\frac 32}\right)D_{n_1,...,n_k}=\det\left((A_{i,j}^{(n)})_{i,j}\right)\quad 
   \mbox{and}\quad \det  \overline M_{t_1,...,t_k}=\det \left(({\mathcal A}_{i,j})_{i,j}\right).$$
Furthermore, for any matrix $M$:
$$\textrm{det}\left((M_{i,j})_{i,j}\right)=\sum_{\sigma\in {\mathcal S}_k}(-1)^{sgn(\sigma)}
      \prod_{i=1}^kM_{i,\sigma(i)},$$
where ${\mathcal S}_k $ is the group of permutations of $\{1,...,k\}$ and
$sgn(\sigma)$ is the signature of $\sigma$.
Therefore, using
(\ref{aij}),
on $F_{n,n_1,\dots,n_k}$, when $n_i\le n$, for all $i\le k$, we get for $n$ large enough,   
\begin{eqnarray*}
\left|\left(\prod_{i=1}^kn_i^{-\frac 32}\right)D_{n_1,...,n_k}
-\det \overline M_{t_1,...,t_k} \right| &\le & 
 \sum_{\sigma\in{\mathcal S}_k}
\sum_{i=1}^k n^{3\gamma(k-1)} 
    \left|A_{i,\sigma(i)}^{(n)}-\A_{i,\sigma(i)}\right| \\
&\le& 4\, (k+1)!\,  n^{3\gamma k}n^{-\frac \theta 4}
 \le  n^{-\theta_0},
\end{eqnarray*}
according to our assumption on $\gamma$.

\noindent Thus, for $n$ large enough, and every $\varepsilon\geq n^{-\theta_0}$, we get by using Proposition \ref{proplambda}
\begin{eqnarray*}
\sup_{n^{\theta}\le n_1,...,n_k\le n}\,{\mathbb P}\left(
    \widetilde \Omega_{n_1,\dots,n_k}^c(\varepsilon)\right)
&\le&
\sup_{n^{\theta}\le n_1,...,n_k\le n} {\mathbb P}(F_{n,n_1,\dots,n_k}^c)+{\mathbb P}\left(
   \det \overline M_{t_1,...,t_k} \le 2\varepsilon\right)\\
  &\le& \O(n^{-\theta_0 L})+{\mathbb P}\left(
   \overline \lambda_{t_1,\dots,t_k}\le (2\varepsilon)^{\frac 1k}\right)\\
  &=&\O\left(\varepsilon^L\right),
\end{eqnarray*}
with $\overline \lambda_{t_1,\dots,t_k}$ as in Proposition \ref{proplambda}. 
So we just proved that for any $L>0$, the constant 
$$ C_L:=\sup_{n\ge 1}\, \sup_{\varepsilon\geq n^{-\theta_0}}\, \varepsilon^{-L} \,
   \sup_{n^{\theta}\le n_1,...,n_k\le n}\, {\mathbb P}\left(
    \widetilde \Omega_{n_1,\dots,n_k}^c(\varepsilon)\right),
$$
is finite, which gives the first part of the lemma. Then we get for any $p>0$, 
\begin{eqnarray*}
&& \sup_{n^{\theta}\le n_1,...,n_k\le n}\,  {\mathbb E}\left[ \left(\prod_{i=1}^k n_i^{\frac {3p}2}\right)D_{n_1,...,n_k}^{-p},\ 
    \widetilde \Omega_{n_1,\dots,n_k}(n^{-\theta_0}) \right]\\
&=& \sup_{n^{\theta}\le n_1,...,n_k\le n}\, \int_0^\infty {\mathbb P}
   \left(n^{-\theta_0}\le 
     \left(\prod_{i=1}^k n_i^{-\frac 32}\right)\, D_{n_1,...,n_k}\le  t^{-1/p}\right)\, dt\\
&=&  \sup_{n^{\theta}\le n_1,...,n_k\le n}\, p\int_{n^{-\theta_0} }^{+\infty}{\mathbb P}
   \left(n^{-\theta_0}\le 
     \left(\prod_{i=1}^k n_i^{-\frac 32}\right)\, D_{n_1,...,n_k}\le \varepsilon \right)\frac{d\varepsilon}{\varepsilon^{p+1}}\\
&\le &p\int_{n^{-\theta_0}}^1  C_{p+1}\, d\varepsilon+p\int_1^{+\infty}\frac
      {d\varepsilon}{\varepsilon^{p+1}}<\infty,
\end{eqnarray*}
where for the third line we have used the change of variables $t= \varepsilon^{-p}$. This concludes the proof of the lemma. 
\end{proof}
\noindent The next step is the 
\begin{lem}\label{lem:gaussian}
Let $\eta\in (0,1/4)$ and $\theta\in (0,1)$ be given. Then 
\begin{eqnarray*}
&& \lim_{n\to \infty}\ \sup_{n^{\theta}\le n_1,...,n_k\le n}\ 
\left(\prod_i n_i^{ 3/4}\right)\\ 
&& \qquad  \times  \int_{U(\eta)}\left| \varphi_{n_1,\dots,n_k}(\theta_1,\dots,\theta_k) - 
 \EE \left[e^{-\sigma^2 Q_{n_1,\dots,n_k}(\theta_1,\dots,\theta_k)/2}  \right]\right|
  \, d\theta_1\dots d\theta_k = 0,
\end{eqnarray*}
where  
$$Q_{n_1,\dots,n_k}(\theta_1,\dots,\theta_k):= \sum_y\left(\theta_1N_{n_1}^{(1)}(y)+\dots +\theta_kN^{(k)}_{n_k}(y)\right)^2.$$
\end{lem}
\begin{proof}
Recall that $U(\eta)=\{ |\theta_i| \le n_i^{-\frac 12-\eta}\ \forall i\le k\}$. 
Set
$$E_{n_1,\dots,n_k}(\theta_1,\dots,\theta_k):=\left(\prod_y\varphi_\xi(\theta_1N_{n_1}^{(1)}(y)+\dots +\theta_kN^{(k)}_{n_k}(y))\right)-
e^{-\sigma^2 Q_{n_1,\dots,n_k}(\theta_1,\dots,\theta_k)/2}. $$
We have to prove that 
$$
 \int_{U(\eta) }
 \EE \left[\left| E_{n_1,\dots,n_k}(\theta_1,\dots,\theta_k) \right|\right]
 \, d\theta_1\dots d\theta_k = o\left(\prod_{i=1}^kn_i^{- 3/4}\right).
$$
Observe that 
\begin{eqnarray*}
E_{n_1,\dots,n_k}(\theta_1,\dots,\theta_k) &= & \sum_y  \left(\prod_{z> y} \exp\left( -\frac{\sigma^2}{2}(\theta_1N_{n_1}^{(1)}(z)+
\dots + \theta_k N_{n_k}^{(k)}(z))^2\right)\right)\\
&\times & \left( \varphi_\xi (\theta_1N_{n_1}^{(1)}(y)+\dots + \theta_k N_{n_k}^{(k)}(y))- 
   e^{-\frac{\sigma^2}{2} (\theta_1 N_{n_1}^{(1)}(y)+\dots +\theta_kN_{n_k}^{(k)}(y) )^2 }\right) \\ 
& \times & \left(\prod_{z< y} \varphi_\xi\left( \theta_1N_{n_1}^{(1)}(z)+
\dots + \theta_k N_{n_k}^{(k)}(z)\right)\right).
\end{eqnarray*} 
Recall now that, since $\xi$ is square integrable,
we have $1-\varphi_\xi(u) \sim \sigma^2|u|^2/2$, as $u\to 0$. It follows that, 
$$\left\vert \varphi_\xi(u)-e^{-\sigma^2 u^2/2}\right\vert\le
|u|^2 h(\vert u\vert) \quad \textrm{for all } u\in \RR,$$ 
with $h$ a continuous and monotone function on $[0,+\infty)$ vanishing in $0$.
In particular there exists a constant $\varepsilon_0>0$, such that
\begin{eqnarray}
\label{majorationphi} 
\vert\varphi_\xi(u)\vert\le
\exp\left(-\sigma^2|u|^2/4\right) \quad \textrm{for all }u\in [-\varepsilon_0,\varepsilon_0].
\end{eqnarray}
Fix now $\gamma \in(0,\eta)$ and $\theta_0\in (0,\theta/4)$. Next recall \eqref{omegatilde} and observe that on 
$$\Omega(\gamma,\theta_0):=\Omega_{n_1,...,n_k}\cap \widetilde \Omega_{n_1,\dots,n_k}(n^{-\theta_0}),$$
if $|\theta_i|\le n_i^{-\frac 12-\eta}$ for all $i\le k$, then 
(see the remark following Lemma \ref{lem:omega_n}) for all $z\in \ZZ$,  
\begin{eqnarray*}
%\label{tnnz}
|\theta_1 N_{n_1}^{(1)}(z) +\dots + \theta_k N_{n_k}^{(k)}(z)| \le kn^{\gamma-\eta},
\end{eqnarray*} 
which is smaller than $\varepsilon_0$ for $n$ large enough. Then we get,
\begin{eqnarray*} 
|E_{n_1,\dots,n_k}(\theta_1,\dots,\theta_k) |{\bf 1}_{\Omega(\gamma,\theta_0)} &\le & 
h(kn^{\gamma-\eta}) e^{- \sigma^2  Q_{n_1,\dots,n_k}(\theta_1,\dots,\theta_k)/4} \\
 & \times & \sum_y e^{\frac{\sigma^2}{4}(\theta_1 N_{n_1}^{(1)}(y) +\dots + \theta_k N_{n_k}^{(k)}(y))^2}   
\left(\theta_1 N_{n_1}^{(1)}(y) +\dots + \theta_k N_{n_k}^{(k)}(y)\right)^2  \\ 
&= &   o(1)\times e^{(\sigma\varepsilon_0)^2}
      e^{-\frac{\sigma^2}{4} Q_{n_1,\dots,n_k}(\theta_1,\dots,\theta_k)} Q_{n_1,\dots,n_k}(\theta_1,\dots,\theta_k)\\ 
&= &  o(1)\times e^{-\frac{\sigma^2}{8} Q_{n_1,\dots,n_k}(\theta_1,\dots,\theta_k)}. 
\end{eqnarray*}
Therefore a change of variables gives
$$
\int_{U(\eta)}
   \left|E_{n_1,\dots,n_k}(\theta_1,\dots,\theta_k) \right|{\bf 1}_{\Omega(\gamma,\theta_0)}\, d\theta_1\dots d\theta_k =o(1)\times   
D_{n_1,\dots,n_k}^{-1/2} \  \int_{{\mathbb R}^k} e^{-\sigma^2 |r|_2^2/8}\, dr, 
$$
at least when $D_{n_1,\dots,n_k}>0$. The result now follows from Lemmas 
\ref{lem:omega_n} and \ref{lem:borne}.
\end{proof}
 
\noindent Finally Proposition \ref{lem:equivalent} is deduced from the following lemma.

\begin{lem}\label{lem:gaussian2}
Let $t_1,\dots,t_k\in(0,1)$ and $\eta\in (0,1/8)$ be given. Then
\[ 
\int_{U(\eta)}
\EE \cro{ e^{- \sigma^2 Q_{n_1,\dots,n_k}(\theta_1,\dots,\theta_k)/2 }} \, 
  d\theta_1\dots d\theta_k
=\left(\frac{\sqrt{2\pi}}{\sigma}\right)^k\,  \C_{t_1,\dots,t_k}\, n^{-3k/4} + o_k(n^{-3k/4}). 
\]
\end{lem}
\begin{proof}
First write 
$$\int_{U(\eta)}
 e^{- \sigma^2 Q_{n_1,\dots,n_k}(\theta_1,\dots,\theta_k)/2} \, d\theta_1\dots d\theta_k = 
I_{n_1,\dots,n_k} - J_{n_1,\dots,n_k},$$
where $I_{n_1,\dots,n_k}$ is the integral over $\RR^k$ and $J_{n_1,\dots,n_k}$ is 
the integral over 
$\{\exists j\ :\ |\theta_j| > n_j^{-\frac 12 -\eta}\}$. A change of variables gives
$$ I_{n_1,\dots,n_k} =\sigma^{-k} D_{n_1,\dots,n_k}^{-1/2}\int_{\RR^k}e^{-|r|_2^2/2}\, dr.$$
According to Proposition \ref{thmconvergencejointe}, we know that
$$n^{-3k/4} D_{n_1,\dots,n_k}\quad  \mathop{\Longrightarrow}^{(\mathcal{L})} \quad   
\widetilde \D_{T_1,\dots,T_k},$$
as $n\to\infty$ and $n_i/n\rightarrow t_i$, for $i=1,...,k$. 
This, combined with Lemma \ref{lem:borne}, shows that 
$$\EE[I_{n_1,\dots,n_k}]=\left(\frac{\sqrt{2\pi}}{\sigma}\right)^k 
       \C_{t_1,\dots,t_k}\ n^{-3k/4}+ o_k(n^{-3k/4}),$$ 
and it just remains to estimate $\EE[J_{n_1,\dots,n_k}]$.

\noindent First consider the matrix $A_{n_1,\dots,n_k}:=(\langle N_{n_i}^{(i)},N_{n_j}^{(j)}\rangle)_{i,j\le k}$, 
and denote by $\mu_{n_1,\dots,n_k}$ its smallest eigenvalue.

\noindent Let now $\theta\in(0,1)$, $0<\theta_0<\frac\theta 4$ 
and $\gamma>0$ be such that $2\eta+ \theta_0+3\gamma(k-1)<1/4$.
We know that on $\Omega_{n_1,\dots,n_k}$, 
$$
\textrm{tr}(A_{n_1,...,n_k})=\sum_{i=1}^k \sum_y N_{n_i}^{(i)}(y)^2 \le kn^{3/2+3\gamma}(1+o_k(1)).$$ 
We deduce that all eigenvalues of $A_{n_1,\dots,n_k}$ are smaller than the right hand side of the above inequality.
In particular on $\widetilde \Omega_{n_1,\dots,n_k}(n^{-\theta_0})$, there exists a constant $c>0$ (depending only on $k$ and the $t_i$'s), 
such that 
\begin{eqnarray*}
\mu_{n_1,...,n_k}&\ge & \frac{D_{n_1,...,n_k}}{(k\, n^{\frac 32 +3\gamma}(1+o_k(1)))^{k-1}} \\ 
    &\ge & c\ n^{\frac 32-\theta_0-3\gamma(k-1)}(1-o_k(1)). 
\end{eqnarray*} 
Then we get 
\begin{eqnarray*}
\mu_{n_1,\dots,n_k} n^{-1-2\eta} &\ge & n^{\frac 12-2\eta-\theta_0-3\gamma(k-1)}(1-o_k(1))\\
&\ge & n^{1/4}(1-o_k(1)),
\end{eqnarray*}
since $2\eta + \theta_0 + 3\gamma(k-1) <1/4$ by hypothesis.  
Note moreover, that 
$$Q_{n_1,\dots,n_k}(\theta_1,\dots,\theta_k) \ge \mu_{n_1,\dots,n_k}\, |(\theta_1,\dots,\theta_k)|_2^2.$$ 
Therefore a change of variables gives
$$J_{n_1,\dots,n_k} \le  \mu_{n_1,\dots,n_k}^{-k/2}  
\int_{\{|r|_2 \ge \mu_{n_1,\dots,n_k}^{1/2}n^{-1/2-\eta}\}}e^{- \sigma^2\, |r|_2^2/2}\, dr,$$
and it then follows from the first part of Lemma \ref{lem:borne} 
that $\EE[J_{n_1,\dots,n_k}] = o_k(n^{-3k/4})$. This concludes the proof of the lemma. 
\end{proof}

\subsection{Proof of Proposition \ref{sec:step1}}
Let $\theta_0\in (0,1/4)$ be fixed. 
Consider the events 
$$H(\theta_1,\dots,\theta_k):=\{|\theta_1N_{n_1}^{(1)}(y)+\dots+\theta_kN_{n_k}^{(k)}(y)| \le \varepsilon_0 \quad \textrm{for all }y\in \ZZ\},$$ 
where $\varepsilon_0$ is as in \eqref{majorationphi}, and 
$$\widetilde H(\theta_1,\dots,\theta_k):=H( \theta_1,\dots,\theta_k)\cap \widetilde \Omega_{n_1,\dots,n_k}(n^{-\theta_0}).$$
Then by using \eqref{majorationphi} and the argument at the end of the proof of 
Lemma \ref{lem:gaussian2} we get 
$$\int_{V(\eta)}\EE\left[\prod_y|\varphi_\xi(\theta_1N_{n_1}^{(1)}(y)+\dots +\theta_kN^{(k)}_{n_k}(y))|,\,  
\widetilde H(\theta_1,\dots,\theta_k)\right]\, d\theta_1\dots d\theta_k = o_k(n^{-3k/4}),$$
and, thanks to Lemma \ref{lem:borne},
\begin{eqnarray*}
&& \sup_{n\ge 1}\ \sup_{n^{\theta}\le n_1,\dots,n_k\le n}\ 
  \left(\prod_{i=1}^kn_i^{\frac 34}\right) \\ 
&& \qquad \times  \int_{V(\eta)}\EE\left[\prod_y|
\varphi_\xi(\theta_1N_{n_1}^{(1)}(y)+\dots +\theta_kN^{(k)}_{n_k}(y))|,\ 
\widetilde H(\theta_1,\dots,\theta_k)\right]\, d\theta_1\dots d\theta_k <\infty.
\end{eqnarray*}
On the other hand by using the H\"older continuity of the local time 
(see Lemma \ref{lem:omega_n}), we get 
$$\PP\left[H(\theta_1,\dots,\theta_k)^c,\ \#\{y\in\ZZ\ :\   |\theta_1N_{n_1}^{(1)}(y)+\dots+\theta_kN_{n_k}^{(k)}(y)|
\in [\varepsilon_0/2,\varepsilon_0]\}\le n^{\frac 14}\right] = o_k(n^{-3k/4}),$$
uniformly in $(\theta_1,...,\theta_k)\in V(\eta)$
and 
\begin{eqnarray*}
&& \sup_{(\theta_1,...,\theta_k)\in V(\eta)}
  \sup_{n\ge 1}\ \sup_{n^{\theta}\le n_1,\dots,n_k\le n}\ 
  \left(\prod_{i=1}^kn_i^{\frac 34}\right) \\ 
&& \ \times\, \PP\left[H(\theta_1,\dots,\theta_k)^c,\ \#\{y\in\ZZ\ :\   
|\theta_1N_{n_1}^{(1)}(y)+\dots+\theta_kN_{n_k}^{(k)}(y)|
\in [\varepsilon_0/2,\varepsilon_0]\}\le n^{1/4}\right]<\infty.
\end{eqnarray*} 
Finally by using again \eqref{majorationphi}, we obtain  
$$\PP\left[H(\theta_1,\dots,\theta_k)^c,\  \left|\prod_y\varphi_\xi(\theta_1N_{n_1}^{(1)}(y)+\dots +\theta_kN^{(k)}_{n_k}(y))\right| 
>  e^{-(\sigma\varepsilon_0/2)^2n^{1/4}/4}\right]= o_k(n^{-3k/4}),$$
and 
\begin{eqnarray*}
&& \sup_{n\ge 1}\ \sup_{n^{\theta}\le n_1,\dots,n_k\le n}\ 
  \left(\prod_{i=1}^kn_i^{\frac 34}\right) \\ 
&& \ \times\, \PP\left[H(\theta_1,\dots,\theta_k)^c,\  \left|\prod_y\varphi_\xi(\theta_1N_{n_1}^{(1)}(y)+\dots +\theta_kN^{(k)}_{n_k}(y))\right| 
>  e^{-(\sigma\varepsilon_0/2)^2n^{1/4}/4}\right]<\infty.
\end{eqnarray*}
The proposition now follows with Lemma \ref{lem:borne}.   
\hfill $\square$

%%%%%%%%%%%%%%%%%%%%%%%%%%%%%%%%%%%%%%%%%%%%%%%%%%%%%%%%%%%%%%%%%%%%%%%%%%%%%%%%%%%%%%%%%%%%%%%%%%%%%%%%%%%%%%%%%%%%%%%%%%%%%%%%%%%%%%%%%%%%%%%%%%
%%%%%%%%%%%%%%%%%%%%%%%%%%%%%%%%%%%%%%%%%%%%%%%%%%%%%%%%%%%%%%%%%%%%%%%%%%%%%%%%%%%%%%%%%%%%%%%%%%%%%%%%%%%%%%%%%%%%%%%%%%%%%%%%%%%%%%%%%%%%%%%%%%
\section{Proof of Corollary \ref{momenttpslocal}}
\label{sectpsloc}
We first observe that for $k=1$, the result follows from \eqref{TLL1b}, since we can write
\begin{eqnarray*}
{\mathbb E}[{\mathcal N}_n(0)]&=&\sum_{i=0}^{n}{\mathbb P}(Z_i=0)
      =\sum_{i=0}^{\left\lfloor n/ d_0 \right\rfloor}
      {\mathbb P}(Z_{id_0}=0) \\
   &\mathop{\sim}_{n\to \infty} &
   \frac d\sigma \sum_{i=0}^{\left\lfloor n/d_0 \right\rfloor}
      p_{1,1}(0)(id_0)^{-\frac 34} \\
   &\sim_{n\to \infty} &   \frac{4d}{\sigma d_0}\, p_{1,1}(0)\, n^{1/4}
  = \frac{d}{\sigma d_0}\, \M_{1,1}(0) \, n^{1/4}, 
\end{eqnarray*}
and
$$ \M_{1,1}(0)=\int_0^1p_{1,t}(0)\, dt=\int_0^1p_{1,1}(0)\, t^{-\frac 34}\, dt=4\, p_{1,1}(0).$$
We prove now the result for some general $k\ge 1$. 
Fix some $\theta\in(0,1/4)$, and write
\begin{eqnarray*}
&& n^{-\frac k4}\, {\mathbb E}\left[{\mathcal N}_n(0)^k\right] =n^{-\frac k4}\sum_{n_1,\dots,n_k\le n}
   {\mathbb P}\left(Z_{n_1}=\dots = Z_{n_k}=0\right)\\
&=& n^{-\frac k 4} \sum_{n_1,\dots,n_k\le \floor {n/d_0}}{\mathbb P}\left(
     Z_{n_1d_0}=\dots =Z_{n_kd_0}=0\right) \\
&=& k!\,   n^{\frac {3k} 4}
   \int_{0\le u_1\le\dots\le u_k\le \frac 1{d_0}} 
  {\mathbb P}\left( Z_{\floor{nu_1}d_0}=\dots =Z_{\floor{nu_k}d_0}=0\right)
   du_1\dots du_k\\
&=&  k!\, n^{\frac {3k} 4}
   \int_{n^{\theta-1}\le u_1\le\dots\le u_k\le \frac 1{d_0}} 
  {\mathbb P}\left( Z_{\floor{nu_1}d_0}=\dots =Z_{\floor{nu_k}d_0}=0\right)
   du_1 \dots du_k+o(1).
\end{eqnarray*}
Indeed, for the last equality, we use Theorem \ref{thmTLL} which implies that for any $\ell \ge 1$, 
\begin{eqnarray*}
&&n^{\frac {3k} 4}
   \int_{0\le u_1\le \dots\le u_\ell\le n^{\theta-1}\le u_{\ell+1}
   \le \dots\le u_k\le \frac 1 {d_0}} 
  {\mathbb P}\left( Z_{\floor{nu_1}d_0}=\dots=Z_{\floor{nu_k}d_0}=0\right)
   du_1\dots du_k\\
&\le &C\, n^{\frac {3\ell}4+(\theta-1)\ell }\, \int_{0\le u_{\ell+1}\le
    \dots \le u_k\le \frac 1{d_0}}(u_{\ell+1}\dots u_k)^{-\frac 34}=o(1), 
\end{eqnarray*}
since $\theta< 1/4$ and 
where $C$ is the constant appearing in Theorem \ref{thmTLL}. 
Then, by using again 
Theorem \ref{thmTLL}, we can apply the Lebesgue dominated 
convergence theorem, and we get
\begin{eqnarray*}
&& n^{-\frac k4}\sum_{n_1,\dots,n_k}
   {\mathbb P}\left(Z_{n_1}=\dots = 
   Z_{n_k}=0\right) \\
&=& k! \left(\frac{d}{\sigma}\right)^k
   \int_{0\le u_1\le u_2\le\dots\le u_k\le 1/d_0}
    p_{k,u_1d_0,u_2d_0,\dots,u_kd_0}(0,\dots,0)\, du+o(1)\\
&=&  \left(\frac{d}{\sigma }\right)^k
   \int_{[0,1/d_0]^k}
    p_{k,u_1d_0,u_2d_0,\dots,u_kd_0}(0,\dots,0)\, du+o(1)\\
&=&\left(\frac{d}{d_0\sigma}\right)^k
   \int_{[0,1]^k}
    p_{k,u_1,u_2,\dots,u_k}(0,\dots,0)\, du+o(1)\\
&=&\left(\frac{d}{d_0\sigma}\right)^k\M_{k,1}(0)+o(1).
\end{eqnarray*}
This concludes the proof of the corollary. \hfill $\square$
 
\vspace{0.2cm}
\noindent We notice that similar calculations show that for any $r\ge 1$, any $k_1,\dots,k_r\ge 1$, and any $0<t_1<\dots<t_r$, 
\begin{eqnarray}
\label{multimoments} 
\EE\left[ \N_{[nt_1]}(0)^{k_1}\dots \N_{[nt_r]}(0)^{k_r}\right] \, \sim \,  \left(\frac{d}{\sigma d_0}\right)^{k_1+\dots +k_r}\, n^{\frac{k_1+\dots+k_r}4}\, 
\EE\left[ \L_{t_1}(0)^{k_1}\dots \L_{t_r}(0)^{k_r}\right],
\end{eqnarray} 
as $n\to \infty$. 

\vspace{0.2cm} 
\noindent At this point, it is also not difficult to see that the sequence $(\N_{[nt]}(0)/n^{1/4},t\ge 0)$ is tight in the Skorokhod space $\DD(\RR)$. 
For this, notice that for any $T>0$ and $p\ge 1$, there exists a constant $C=C(T,p)$, such that 
for all $t\in [0,T]$, $h>0$, and $\eta>0$,   
\begin{eqnarray*}
\PP\left( \N_{[n(t+h)]}(0) - \N_{[nt]}(0) \ge \eta\, n^{1/4} \right)&\le & \eta^{-p}\, n^{-p/4}\, \EE\left[ \left( \N_{[n(t+h)]}(0) - \N_{[nt]}(0) \right)^p\right] \\ 
&\le & C\, \eta^{-p} h^{p/4}.
\end{eqnarray*}  
Indeed the second inequality follows from the proof of Corollary \ref{momenttpslocal}. 
Since $\N_{[nt]}(0)$ is a nondecreasing process, the tightness follows for instance from Lemma (1.7) p.517 in \cite{RY}.

%%%%%%%%%%%%%%%%%%%%%%%%%%%%%%%%%%%%%%%%%%%%%%%%%%%%%%%%%%%%%%%%%%%%%%%%%%%%%%%%%%%%%%%%%%%%%%%%%%%%%%%%%%%%%%%%%%%%%%%%%%%%%%%%%%%%%%%%%%%%%%%%%
%%%%%%%%%%%%%%%%%%%%%%%%%%%%%%%%%%%%%%%%%%%%%%%%%%%%%%%%%%%%%%%%%%%%%%%%%%%%%%%%%%%%%%%%%%%%%%%%%%%%%%%%%%%%%%%%%%%%%%%%%%%%%%%%%%%%%%%%%%%%%%%%%

\section{Proof of Proposition \ref{thmconvergencejointe}}\label{preuvecvgcejointe}
It was proved by Kesten and Spitzer \cite{KestenSpitzer} that the normalized self-intersection local time of the random walk 
converges in distribution to its continuous counterpart. 
A similar convergence is proved for 
the mutual intersection local time in Chen's book \cite{BookChen}.
We prove Proposition~\ref{thmconvergencejointe}
by following carefully their proof.

\noindent For $j=1,...,k$, and $a<b$, let
 $$T_{n_j}^{(j)}(a,b):=\frac{1}{n_j}\sum_{a\leq
n^{-\frac{1}{2}}y<b}N_{n_j}^{(j)}(y),$$
denotes the time spent by $S_{\lfloor n_j\cdot\rfloor}
   ^{(j)}/\sqrt{n}$ in $[a,b)$ before time $n_j$. 
The mutual intersection local time of  
$S^{(j)}_{[n_j\cdot]}/\sqrt{n}$ 
and $S^{(j')}_{[n_{j'}\cdot]}/\sqrt{n}$ before time $1$ is defined by:
\begin{eqnarray*}
T_{n_j,n_{j'}}^{(j,j')}&:= &
\frac{\sqrt{n}}{n_jn_{j'}}\langle N_{n_j}^{(j)},N_{n_{j'}}^{(j')}\rangle\\
&=&\frac{\sqrt{n}}{n_jn_{j'}} \sum_{x\in \mathbb{Z}} \sum_{k=1}^{n_j}
   \sum_{\ell=1}^{n_{j'}} {\mathbf 1}_{\{S_k^{(j)}=x\}}{\mathbf 1}_{\{S_\ell^{(j')}=x\}}.
\end{eqnarray*}
For any $\varepsilon>0$, consider the regularizing functions 
$p_{\varepsilon}(x) :=e^{-x^2/ 2\varepsilon}/\sqrt{2\pi \varepsilon}$,  
and set 
\begin{eqnarray*}
T_{\varepsilon,n_j,n_{j'}}^{(j,j')}&:=& \frac{1}{\sqrt{n}n_jn_{j'}}
     \sum_{x\in \mathbb{Z}} \sum_{k=1}^{n_j}
   \sum_{\ell=1}^{n_{j'}}  p_{\varepsilon}\Big(\frac{S_k^{(j)} 
      -x}{\sqrt{n}}\Big) p_{\varepsilon}
       \Big(\frac{{S}_\ell^{(j')} -x}{\sqrt{n}}\Big).
\end{eqnarray*}
Similarly, let 
$$
\Lambda_{j}(a,b):=\frac 1{T_j}\int_{a}^{b}L^{(j)}_{T_j}(x)\, dx,$$
denotes the time spent by $B^{(j)}$ in $[a,b)$ before time $T_j$, and 
let 
$$\Lambda_{j,j'}:=\frac 1{T_jT_{j'}}
  \int_{\mathbb R} L_{T_j}^{(j)}(x) {L}_{T_{j'}}^{(j')}(x) \, dx,$$
denotes the mutual intersection local time of $B^{(j)}_{T_j\cdot}$ and $B^{(j')}_{T_{j'}\cdot}$. 
Finally set for every $\varepsilon > 0$,
$$\Lambda_{\varepsilon,j,j'}: =\int_{\mathbb R}\left( \int_{[0,1]^2}
   p_{\varepsilon}(B_{s T_j} -x) 
 p_{\varepsilon}({B}_{t T_{j'}} -x) \, ds\, dt\right)\, dx.$$
We will use the following lemmas: 
\begin{lem}\label{approx1}  (Lemma 5.3.1 in Chen ) For all $j\ne j'$, 
$$\lim_{\varepsilon\rightarrow 0}\, \limsup_{n\to \infty}\, \limsup_{n_j/n\to T_j,\, 
n_{j'}/n \to T_{j'}} \mathbb{E}
\left[ |T^{(j,j')}_{n_j,n_{j'}} - T_{\varepsilon,n_j,n_{j'}}^{(j,j')}|^2\right] = 0.$$
\end{lem}

\begin{lem}\label{approx2}  (Theorem 2.2.3 in Chen)
For all $j\ne j'$,
The sequence $ (\Lambda_{\varepsilon,j,j'},\varepsilon>0)$ converges in $L^2$ to  
$\Lambda_{j,j'}$, as $\varepsilon$ goes to 0.
\end{lem}
\noindent We can then already deduce the following:       
\begin{lem}\label{pre2} 
For any $m_1,\ldots,m_k\ge 1$ and any $-\infty<a_{j,\ell}<b_{j,\ell}<\infty$ 
($j=1,\ldots,k$ and $\ell=1,\ldots,m_j$), 
$$
\Big(\left(T_{n_j}^{(j)}(a_{j,\ell},b_{j,\ell})\right)_{j=1,\dots,k,\,
 \ell=1,\dots,m_j},
\left(T_{n_j,n_{j'}}^{(j,j')}\right)_{1\le j<j'\le k} \Big),$$
converges in distribution to 
$$ \Big(\left(\Lambda_{j}(a_{j,\ell},b_{j,\ell})\right)_{j=1,\dots,k,\, \ell=1,\dots,m_j},
\left(\Lambda_{j,j'}\right)_{1\le j<j'\le k} \Big),$$
as $n \to +\infty$, and $n_j/n\rightarrow T_j$ for all $j\le k$. 
\end{lem}
\begin{proof}[Proof of Lemma \ref{pre2}] 
Let $\theta_{j,\ell}$ (for $j=1,\dots,k$ and $\ell=1,\dots,m_j$) and
$\overline\theta_{j,j'}$ (for $1\le j<j'\le k$) be some fixed real numbers.
It suffices to prove that
$$ \mathbb{E}\left( \exp\left(i \sum_{j=1}^k\sum_{\ell=1}^{m_j}\theta_{j,\ell}
   T_{n_j}^{(j)}(a_{j,\ell},b_{j,\ell}) + i\sum_{1\le j< j'\le k}
   \overline{\theta}_{j,j'}  T_{n_j,n_{j'}}^{(j,j')} \right)\right),$$
converges to
$$\mathbb{E}\left( \exp\left(i \sum_{j=1}^k\sum_{\ell=1}^{m_j}\theta_{j,\ell}
\Lambda_{j}(a_{j,\ell},b_{j,\ell})+i \sum_{1\le j< j'\le k}
   \overline{\theta}_{j,j'} \Lambda_{j,j'} \right)\right) .$$
Lemmas \ref{approx1} and \ref{approx2} show that we can replace 
the $T_{n_j,n_{j'}}^{(j,j')}$ and $\Lambda_{j,j'}$, respectively by 
$T_{\varepsilon,n_j,n_{j'}}^{(j,j')}$ and $\Lambda_{\varepsilon,j,j'}$.

\noindent Observe now that the map
\begin{eqnarray*}
(x^{(j)})_{j\le k} &\mapsto & \sum_{j=1}^k\sum_{\ell=1}^{m_j}
    \theta_{j,\ell} \int_{0}^{1} \mbox{\bf 1}_{[a_{j,\ell}\leq x^{(j)}_{s} < b_{j,\ell}]} \, ds \\
    && +\sum_{1\le j<j'\le k} \overline{\theta}_{j,j'} \int_{\mathbb R} \int_{[0,1]^2} 
    p_{\varepsilon} (x_s^{(j)}- x) p_{\varepsilon} ({x}^{(j')}_t-x) \, ds\, dt \, dx,
\end{eqnarray*}
 is  continuous  on ${\DD}([0,1],\mathbb R^k)$ for the Skorokhod topology.
Observe moreover, that for all fixed $\varepsilon >0$, 
$$T_{\varepsilon,n_j,n_{j'}}^{(j,j')}=\int_{\mathbb R} \int_{[0,1]^2} 
    p_{\varepsilon} \left(\frac{S^{(j)}_{[n_js]}}{\sqrt{n}}-x \right)
   p_{\varepsilon} \left(\frac{\tilde{S}_{[n_{j'}t]}}{\sqrt{n}}-x\right)\, ds\, dt\, dx + o(1).$$
Therefore the weak convergence of
$\left(S_{[n_j\cdot]}^{(j)}/\sqrt{n},j\le k\right)
$ toward 
$(B^{(j)}_{T_j\cdot}, j\le k)$, implies that 
$$\sum_{j=1}^k \sum_{\ell=1}^{m_j}\theta_{j,\ell} T_{n_j}^{(j)}(a_{j,\ell},b_{j,\ell}) 
+\sum_{1\le j<j'\le k}\overline{\theta}_{j,j'}T_{\varepsilon,n_j,n_{j'}}^{(j,j')}, $$
converges in distribution to 
$$ \sum_{j=1}^k\sum_{\ell=1}^{m_j}\theta_{j,\ell} \Lambda_j(a_{j,\ell},b_{j,\ell})
  +\sum_{1\le j<j'\le k}\overline{\theta}_{j,j'} \Lambda_{\varepsilon,j,j'}.$$ 
The result follows. 
\end{proof}
\noindent We finish now the proof of Proposition \ref{thmconvergencejointe}. 
Let $\theta_{j}$ (for $j=1,\dots,k$) and 
$\theta_{j,j'}$ (for $1\le j<j'\le k$) be some fixed real numbers. 
We proceed like in \cite{KestenSpitzer}  by decomposing the set of all possible
indices into small slices where sharp estimates can be made. 
Define, in the slice $[\tau\ell\sqrt{n},\tau(\ell+1)\sqrt{n})$, an
average occupation time by
$$
T_j(\tau,\ell,n):=\frac {n_j}n\, T_{n_j}^{(j)}(\tau\ell,\tau(\ell+1)).
$$
Set also
$$U(\tau,M,n):=\sum_{j=1}^k\theta_j n^{-\frac{3}{2}}
   \sum_{|x|\ge M\tau\sqrt n}N_{n_j}^{(j)}(x)^2,$$
$$V(\tau,M,n):=\sum_{j=1}^k\frac{\theta_j}{\tau}\sum_{-M\leq
\ell<M}(T_j(\tau,\ell,n))^2  + \sum_{1\le j<j'\le k} \theta_{j,j'}
\frac{n_jn_{j'}}{n^2}\, T_{n_j,n_{j'}}^{(j,j')}
 ,$$
and
\begin{eqnarray*}
A(\tau, M,n)&:=&\sum_{j=1}^k\theta_j n^{-\frac{3}{2}}\sum_{x\in\ZZ} N_{n_j}^{(j)}(x)^2
+ \sum_{1\le j<j'\le k} \theta_{j,j'}\frac{n_jn_{j'}}{n^2}\, T_{n_j,n_{j'}}^{(j,j')}
 -U(\tau,M,n)-V(\tau,M,n)\\
& =&\sum_{j=1}^k\theta_j n^{-\frac{3}{2}}\sum_{-M\leq \ell<M}\sum_{a(\ell,n)\leq x<a(\ell+1,n)}
\left(N^{(j)}_{n_j}(x)^2-\frac{n^2\times(T_j(\tau,\ell,n))^2}{(\tau\sqrt{n})^2}\right)+o(1).
\end{eqnarray*}
It follows from computations in \cite{KestenSpitzer} (see in particular Lemmas 1, 2 and 3) that $A(\tau, M,n)$ converges in
probability to zero as $M\tau^{3/2}\to 0$. Moreover, 
$$\PP(U(\tau,M,n)\neq 0) \le \PP\left(\exists j\le k\ :\ \sup_{m\le n_j} |S^{(j)}_m|> M\tau \sqrt n\right),$$
and it is well known that the right hand term goes to $0$, as $M\tau\to \infty$, and $n_j/n\to T_j$, for all $j\le k$.

\noindent Now, Lemma \ref{pre2} shows that $V(\tau,M,n)$ converges
in law to
$$\sum_{j=1}^k\frac{\theta_j}{\tau}\sum_{-M\leq  \ell<M}\Big(\int_{\ell\tau}^{(\ell+1)\tau}
   L_{T_j}^{(j)}(x)dx\Big)^2
   +\sum_{1\le j<j'\le k}\theta_{j,j'}\int_{\mathbb R}L_{T_j}^{(j)}(x)
     L_{T_{j'}}^{(j')}(x) \, dx.$$
But the map $x\mapsto L^{(j)}_{t}(x)$ being a.s. continuous with compact support, this last sum
converges, as $\tau\rightarrow 0$ and $M\tau\rightarrow\infty$, to
$$\sum_{j=1}^k{\theta_j}\int_{\mathbb R} L_{T_j}^{(j)}(x)^2\,   dx
  +\sum_{1\le j<j'\le k}\theta_{j,j'}\int_{\mathbb R}L_{T_j}^{(j)}(x)
     L_{T_{j'}}^{(j')}(x) \, dx.$$
The proposition follows.  
\hfill $\square$

\end{document}